\documentclass[12pt]{article}

\usepackage{graphicx, verbatim}

\usepackage{amsfonts}
\usepackage{amssymb}
\usepackage{amstext}
\usepackage{amsmath}
\DeclareMathOperator*{\bigplus}{\scalerel*{+}{\sum}}
\usepackage{amsthm, verbatim, scalerel}
\usepackage{xspace}
\usepackage{graphicx}

\newtheorem{theorem}{Theorem}
\newtheorem{lemma}[theorem]{Lemma}
\newtheorem{conjecture}[theorem]{Conjecture}

\newcommand{\E}{\ensuremath{\mathbb E}}
\newcommand{\R}{\ensuremath{\mathbb R}}

\newcommand{\F}{\ensuremath{\psi}}

\newcommand{\tw}{\tilde{w}}
\newcommand{\lab}{\label}  \newcommand{\ra}{\ensuremath{\rightarrow}}  \def\a{{\mathbf{\alpha}}} \def\de{{\mathbf{\delta}}} \def\De{{{\Delta}}}  
 \def\var{{\mathrm{var}}} \def\beq{\begin{eqnarray}} \def\eeq{\end{eqnarray}} \def\ben{\begin{enumerate}}
\def\een{\end{enumerate}} \def\bit{\begin{itemize}}
\def\bel{\begin{lemma}}
\def\eel{\end{lemma}}
\def\eit{\end{itemize}} \def\beqs{\begin{eqnarray*}} \def\eeqs{\end{eqnarray*}} \def\bel{\begin{lemma}} \def\eel{\end{lemma}}
\newcommand{\N}{\mathbb{N}} \newcommand{\Z}{\mathbb{Z}}  \newcommand{\C}{\mathcal{C}} \newcommand{\CC}{\mathcal{C}}
\newcommand{\T}{\mathbb{T}} \newcommand{\A}{\mathbb{A}}     \renewcommand{\b}{\mathbf{b}} 
\newcommand{\tb}{\tilde{\mathbf{b}}}
  \newcommand{\II}{\mathcal{I}}  \newcommand{\p}{\mathbb{P}}
   \newcommand{\e}{\mathbf{e}} 
\newcommand{\LL}{\Delta}  \newcommand{\la}{\lambda}  
   \def\eps{{\epsilon}} \def\A{{\mathcal{A}}} \def\ie{i.\,e.\,} 
\def\vol{\mathrm{vol}}
\newcommand{\tP}{\tilde{P}}

\newcommand{\wn}{w^{(n)}}
\newcommand{\wone}{w^{(n_1)}}

\newcommand{\La}{\Lambda}

\renewcommand{\L}{\mathbb{L}} 
\renewcommand{\i}{k}
\renewcommand{\j}{\ell}

\renewcommand{\L}{\mathbb L}
\renewcommand{\C}{\mathcal C}
\newcommand{\f}{\mathbf{f}}

\newtheorem{thm}{Theorem}[section]
\newtheorem{cor}[thm]{Corollary}
\newtheorem{lem}[thm]{Lemma}

\newtheorem{claim}[thm]{Claim}
\theoremstyle{definition}
\newtheorem{defn}[thm]{Definition}
\theoremstyle{remark}

\numberwithin{equation}{section}

\usepackage{chngcntr}

\begin{document}

\title{Random discrete concave functions on an equilateral lattice with periodic Hessians}


\author{Hariharan Narayanan\\
School of Technology and Computer Science\\
Tata Institute for Fundamental Research, \\
Mumbai 400005, India.\\
hariharan.narayanan@tifr.res.in}
\maketitle
\begin{abstract}
Motivated by connections to random matrices, Littlewood-Richardson coefficients and tilings, we study random discrete concave functions on an equilateral lattice. We show that such functions having a periodic Hessian of a  fixed average value $- s = - (s_0, s_1, s_2) \in \R_{< 0}^3$ concentrate around a quadratic function under certain conditions.
We consider the set of all discrete concave functions $g$ (\ie functions whose piecewise linear extensions are concave) on an equilateral lattice $\mathbb L$ that when shifted by an element of $n \mathbb L$ have a periodic discrete Hessian, with period $n \mathbb L$. 
We add a convex quadratic of Hessian $s$; the sum is then periodic with period $n \mathbb L$, and view this as a mean zero function $g$ on the set of vertices $V(\T_n)$ of a torus $\T_n := \frac{\Z}{n\Z}\times \frac{\Z}{n\Z}$ whose Hessian is dominated by $s$. The resulting set of semiconcave functions forms a convex polytope $P_n(s)$. 
The $\ell_\infty$ diameter of $P_n(s)$ is shown to be bounded below by $c(s) n^2$, where $c(s)$ is a positive constant depending only on $s$. 
We show that the surface tension $\sigma(s) = - \lim_{n\ra \infty} \left(\frac{1}{n^2}\right)\log |P_n(s)|$ is well defined and convex; in fact that $\exp(- \sigma(s))$ is concave.
Our main result is that when $s$ is such that a subgradient $w = (w_0, w_1, w_2)$ of $\sigma(s)$ belongs to the cone
\beqs w_0^2 + w_1^2 + w_2^2 < 2\left(w_0 w_1 + w_1 w_2 + w_2 w_0\right),\eeqs (which happens to be true for when $s_0 = s_1 \leq s_2$,)  then for any $\eps > 0,$ 
 $$\lim_{n \ra 0} \p\left[\|g\|_\infty > n^{\frac{7}{4} + \eps}\right] = 0$$  where $g$ is sampled from the uniform measure on $P_n(s)$. We also prove concentration bounds 
if the surface tension at $s$ is strictly convex. Each $g \in P_n(s)$ corresponds to a kind of  honeycomb. We obtain concentration results for these as well. 
Along the way, we provide an upper bound on the volume of $P_n(s)$. This bound involves the determinant of a Laplacian on the torus. 
\end{abstract}



%
%
\tableofcontents

\section{Introduction}\lab{sec:intro}

\subsection{Motivation from Littlewood-Richardson coefficients}
Littlewood-Richardson coefficients play an important role in the representation theory of  the general linear groups. Among other interpretations, they count the number of tilings of certain domains using squares and equilateral triangles  \cite{squaretri}. 
Let $\la, \mu, \nu$ be vectors in $\Z^n$ whose entries are non-increasing non-negative integers. Let the $\ell_1$ norm of a vector $\a \in \R^n$ be denoted $|\a|$ and let $$|\la| + |\mu| = |\nu|.$$ Take an equilateral triangle $\Delta$ of side $1$. Tessellate it with unit equilateral triangles of side $1/n$. Assign boundary values to $\Delta$ as in Figure~\ref{fig:tri3}; Clockwise, assign the values $0, \la_1, \la_1 + \la_2, \dots, |\la|, |\la| + \mu_1, \dots, |\la| + |\mu|.$ Then anticlockwise, on the horizontal side, assign  $$0, \nu_1, \nu_1 + \nu_2, \dots, |\nu|.$$

Knutson and Tao defined this hive model for Littlewood-Richardson coefficients in \cite{KT1}. They showed that the Littlewood-Richardson coefficient
$c_{\la\mu}^\nu$ is given by the number of ways of assigning integer values to the interior nodes of the triangle, such that the piecewise linear extension to the interior of $\Delta$ is a concave function $f$ from $\Delta$ to $\R$.  
Such  an integral ``hive" $f$ can be described as an integer point in a certain polytope known as a hive polytope. The volumes of these polytopes shed light on the asymptotics of Littlewood-Richardson coefficients \cite{Nar, Okounkov, Greta}. Additionally, 
they appear in certain calculations  in free probability \cite{KT2, Zuber}. Indeed,
the volume of the polytope of all real hives with fixed boundaries $\la, \mu, \nu$ is equal, up to known multiplicative factors involving Vandermonde determinants, to the probability density of obtaining a Hermitian matrix with spectrum $\nu$ when two Haar random Hermitian matrices with spectra  $\la$ and $\mu$   are added \cite{KT2}. 

\begin{figure}\label{fig:tri3}
\begin{center}
\includegraphics[scale=0.40]{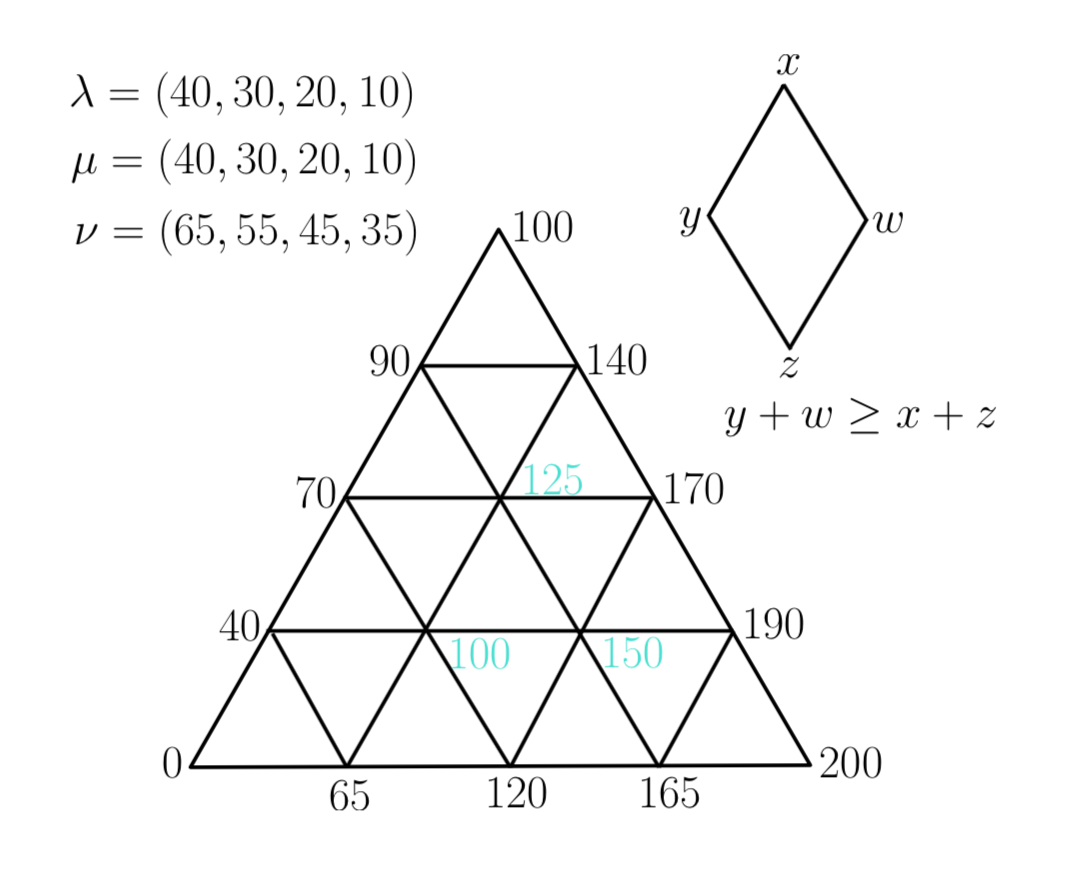}
\caption{Values taken at interior vertices in the hive model}
\end{center}
\end{figure}
Corresponding to every real hive, is a gadget known as a honeycomb, which is a hexagonal tiling. The positions of the lines corresponding to the semi-infinite rays are fixed by the boundary data $\la, \mu$ and $\nu$, with each segment being parallel to one of the sides of a regular hexagon. One obtains a random honeycomb from a random hive by mapping the gradient of the hive on each of the unit equilateral triangles to a point in $\R^2$. This point becomes a vertex of the honeycomb.
\begin{figure}\label{fig:honey}
\begin{center}
\includegraphics[scale=0.80]{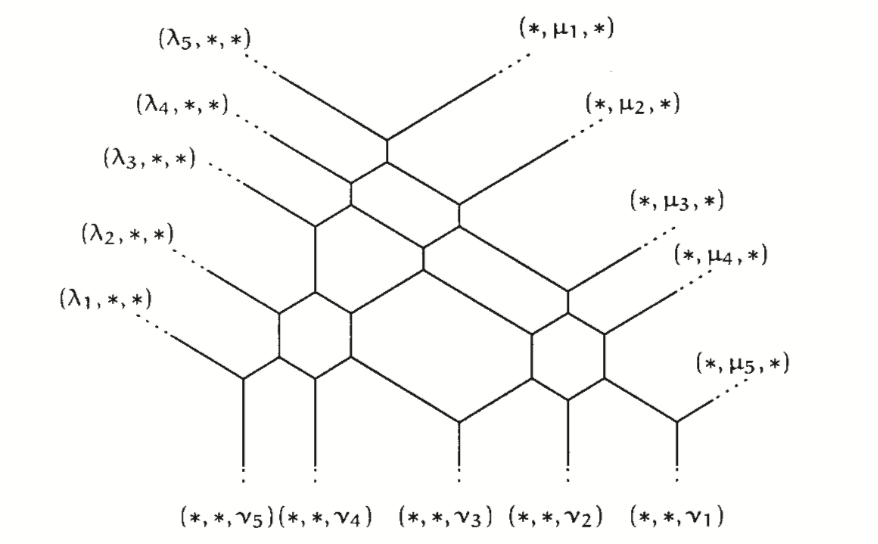}
\caption{A honeycomb, from Knutson and Tao \cite{KT2}}
\end{center}
\end{figure}

The question of studying the structure of a typical real hive in a hive polytope, sampled from the Lebesgue measure is closely linked to the question of evaluating the asymptotic value of a Littlewood-Richardson coefficient for $GL_n(\mathbb C)$ as $n \ra \infty$ and $\la, \mu$ and $\nu$ tend to continuous monotonically decreasing functions in a certain fashion.
In order to study the scaling limits of random surfaces \cite{Scott}, it has proven beneficial to first examine the situation with periodic boundary conditions \cite{CohnKenyonPropp}. These structures correspond to random periodic honeycombs, where the periodicity is at a scale that tends to infinity. The results of this paper give the first results on concentration phenomena for these objects (see Subsection~\ref{ssec:honey}).

\subsection{Overview}

We consider the set of all (discrete) concave functions on an equilateral lattice $\mathbb L$ that when shifted by an element of $n \mathbb L$, incur addition by a  linear function (this condition is equivalent to the periodicity of the Hessian). We subtract a quadratic of the same Hessian $- s$; the difference is then periodic with period $n \mathbb L$, and view this as a mean zero function $g$ on the vertices $V(\T_n)$ of a torus $\T_n := \frac{\Z}{n\Z}\times \frac{\Z}{n\Z}$ whose Hessian is bounded above by $s$. The resulting set of functions forms a convex polytope $P_n(s)$. 
 We show in Lemma~\ref{lem:diameter} that 
the $\ell_\infty$ diameter of $P_n(s)$ is bounded below by $c(s) n^2$, where $c(s)$ is a positive constant depending only on $s$. 
We prove an upper bound on the differential entropy per vertex in terms of a determinant of a Laplacian. 
Suppose exists a superdifferential $w$ of $\mathbf{f}$ at $s$ such that 
\beq w_0^2 + w_1^2 + w_2^2 < 2\left(w_0 w_1 + w_1 w_2 + w_2 w_0\right).\lab{eq:w}\eeq
 We show  in Theorem~\ref{thm:2} that concentration in $\ell_\infty-$norm takes place for average Hessian $s$ if there exists a superdifferential $w$ of $\mathbf{f}$ at $s$ such that (\ref{eq:w}) holds.
Theorem~\ref{thm:2}, also provides quantitative bounds, namely that for any positive $\eps$, $$\lim_{n \ra 0} \p\left[\|g\|_\infty > n^{\frac{7}{4} + \eps}\right] = 0,$$  if $g$ is sampled from the uniform measure on $P_n(s)$.

In the rest of this section, we outline the main ideas that go into proving  Theorem~\ref{thm:2}.
We first show that $\f_n(s) := |P_n(s)|^\frac{1}{n^2 - 1}$ tends to a limit $\f(s)$ (which by the Brunn-Minkowski inequality is concave) as $n \ra \infty$, and further, 
that there is a universal constant $C > 0$ such that for all $n \geq 2$, $$ \left|\frac{\f_n(s)}{\f(s)} - 1\right| \leq \frac{ C\ln n}{n}.$$ We identify a convex set $K \subseteq \R^{V(\T_n)}$ consisting of ``tame" functions whose discrete $L_2^2$ Sobolev seminorm and discrete $\dot{C}^2$ seminorm are respectively bounded above by certain quantities. We show that the probability measure of $P_n(s) \setminus K$ is negligible, and focus our attention on $P_n(s) \cap K$. We appeal to a theorem of Bronshtein, which states that 
the set of Lipschitz, bounded convex functions on a bounded domain of dimension $d$, can be covered using $\exp(C \eps_{0.5}^{-\frac{d}{2}})$ $L^\infty-$balls of radius $\eps_{0.5}$. It follows that $P_n(s) \cap K$ can be covered by $\exp(C \eps_{0.5}^{-1})$  $\ell_\infty-$balls of radius $\eps_{0.5} n^2.$ Without loss of generality, doubling $\eps_{0.5}$ if necessary, we may assume that these balls are all centered in $P_n(s) \cap K$. 

We next let $\rho$ be a probability measure on $V(\T_n)$, and prove using the Brunn-Minkowski inequality that for any $g \in \R^{V(\T_n)}$, the measure of the $l_\infty-$ball $B_\infty(g, \eps_{0.5}n^2)$ of radius $\eps_{0.5}$ around $g$ is less or equal to the measure of $B_\infty(\rho\ast g, \eps_{0.5}n^2)$, where $\ast$ denotes convolution on $\T_n$. For $g \in P_n(s) \cap K$, we show the existence of a character $\psi_{k_0\ell_0}$ of $\T_n$, where $\psi_{k_0\ell_0}(i, j) = \exp\left(\frac{2\imath\pi(k_0 i + \ell_0 j)}{n}\right)$ such that, if we set $\rho$ to $ \frac{2 + \psi_{k_0\ell_0} + \psi_{-k_0\, - \ell_0}}{2},$ then $\rho \ast g = \Re(\theta_{k_0\ell_0}\psi_{k_0\ell_0})$ is the real part of a complex exponential of low frequency and large amplitude. This allows us to reduce our problem to one of bounding from above the probability measure of $B_\infty(\Re(\theta_{k_0\ell_0}\psi_{k_0\ell_0}), \eps_{0.5}n^2)$, where we have an a priori lower bound on $\theta_{k_0\ell_0}$ and 
an a priori upper bound on $k_0^2 + \ell_0^2.$ 

In order to do this, we partition $\T_n$ into squares $\square_{ij}$ of sidelength $n_1 \approx \eps_1 n$, with a small residual set of vertices and define a set of boundary vertices $\b$ that is the union of all the sides of all the squares. This is a ``double layer" boundary, and conditioning on the values taken by $g$ on $\b$ , results in the values taken by $g$ on the different $\square_{ij}$ being completely decoupled. In particular, this allows us to bound from above the measure of $B_\infty(\Re(\theta_{k_0\ell_0}\psi_{k_0\ell_0}), \eps_{0.5}n^2)\cap K $, by an integral over $\R^{\b}$ of the product of the measures of the projections on the different $\R^{\square_{ij}},$ of certain sections of this polytope defined by the conditioning. The Hessian of $\Re(\theta_{k_0\ell_0}\psi_{k_0\ell_0})$ varies from point to point as a scalar multiple of a fixed vector that is very close to $(k_0(k_0 + \ell_0), -k_0\ell_0, \ell_0(k_0 + \ell_0)),$ in $\R^3$. Using an inequality involving the anisotropic surface area of a convex set that can be derived from the Brunn-Minkowski inequality, we obtain from the above product of measures of projections, a more convenient upper bound of the form $$n^{\frac{4n^2}{n_1^2}}\prod\limits_{1 \leq i, j \leq \frac{n }{n_1}}   |P_{n_1}(t_{ij})|,$$ where $t_{ij} - s$ are, roughly speaking, average Hessians of $\Re(\theta_{k_0\ell_0}\psi_{k_0\ell_0})$ on the respective squares $\square_{ij}$. Theorem~\ref{thm:2} now follows from an inequality relating $\f(t_{ij})$, $\f(s)$, the superdifferential $\partial \f(s)$, and a lower bound on the defect $$\f(t_{ij}) - \f(s) - (\nabla \f(s)) \cdot (t_{ij} - s).$$
This lower bound follows from the quadratic inequality involving $w_0, w_1$ and $w_2$ in (\ref{eq:w}) via a discriminant computation. More specifically, we show that (\ref{eq:w}) implies that  $(w_0, w_1, w_2) \cdot (k_0(k_0 + \ell_0), -k_0\ell_0, \ell_0(k_0 + \ell_0))$ is bounded away from zero, with some quantitative control on this. This essentially corresponds to proving the strict concavity of $\log \f$ at any $s$ in the directions that matter, though this is not estabilished in general. 
In fact, at points $s$ such that that the surface tension $\sigma(s) = - \log \f$ is strictly convex, Theorem~\ref{thm:1} shows that concentration occurs for a random point in $P_n(s)$ with respect to the $\ell_\infty$ norm, by a simpler argument.

\begin{figure}\label{fig:redcube}
\begin{center}
\includegraphics[scale=0.4]{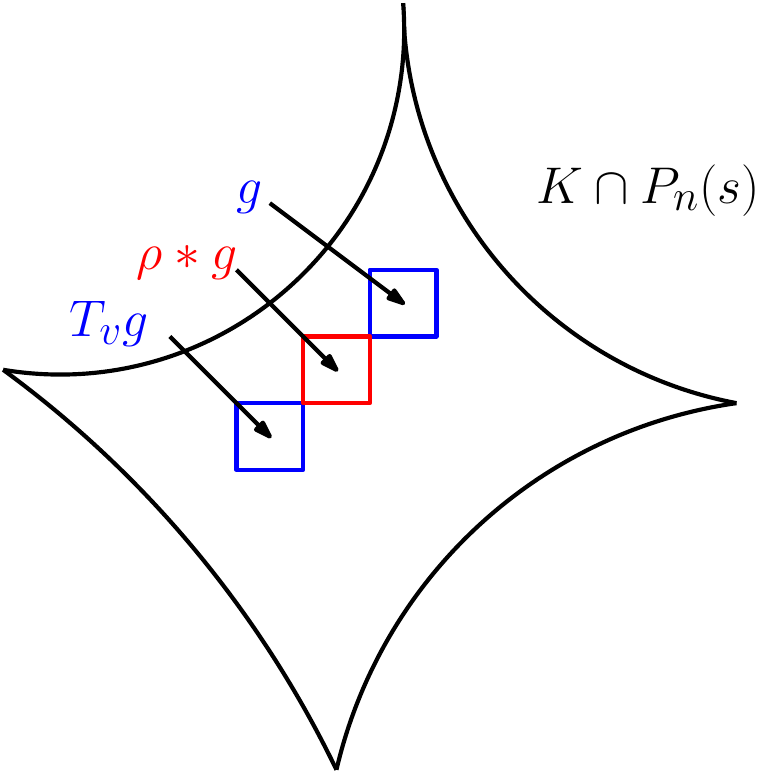}
\caption{
The volume of the intersection of the cube centered at $\rho \ast g$ with $K \cap P_n(s)$ is at least as much as the volume of the intersection of the  cubes centered at $g$ and $T_v g$ with $K \cap P_n(s)$. The function $\rho\ast g$ is a highly structured. It is the real part of a complex exponential.}
\end{center}
\end{figure}

\section{Preliminaries}\lab{sec:prelim}

We consider the equilateral triangular lattice $\mathbb L$, \ie the subset of $\mathbb{C}$ generated by $1$ and $\omega = e^{\frac{2\pi \imath}{3}}$ by integer linear combinations. 
We define the edges  $E(\mathbb L )$ to be  the lattice rhombi of side $1$ in $\mathbb L$. 
We consider a rhombus  $R_n$ with vertices $0$, $n$, $n (1 - \omega^2)$ and $-n\omega^2$.  Let $\T_n$ be the torus obtained from $R_n$  by identifying opposite sides together. We define the (hyper)edges  $E(\T_n)$ to be  the lattice rhombi of side $1$ in $\T_n$, where each vertex in $V(\T_n)$ is  an equivalence class of $\mathbb{L}$ modulo  $n\mathbb{L}:= n\Z   + n \omega\Z$. 

\begin{defn}[Discrete Hessian]
Let $f:\mathbb{L} \ra \R$ be a function defined on $\mathbb L$.
We define the (discrete) Hessian $\nabla^2(f)$ to be a  function from the set $E(\T_n)$ of rhombi of the form $\{a, b, c, d\}$ of side $1$ (where the order is anticlockwise, and the angle at $a$ is $\pi/3$) on the discrete torus to the reals, satisfying 
$$\nabla^2 f(\{a,b,c,d\}) =- f(a) + f(b) - f(c) + f(d).$$  
\end{defn}

Let $f$ be a function defined on $\mathbb{L}$ such that $\nabla^2(f)$ is periodic modulo $n\mathbb{L}$ and the piecewise linear extension of $f$ to $\mathbb C$ is concave. Such a function $f$ will be termed concave on $\L$, or simply concave. Then $\nabla^2(f)$ may be viewed as a function $g$ from $E(\T_n)$ to $\R$.

\begin{figure}
\begin{center}
\includegraphics[scale=0.4]{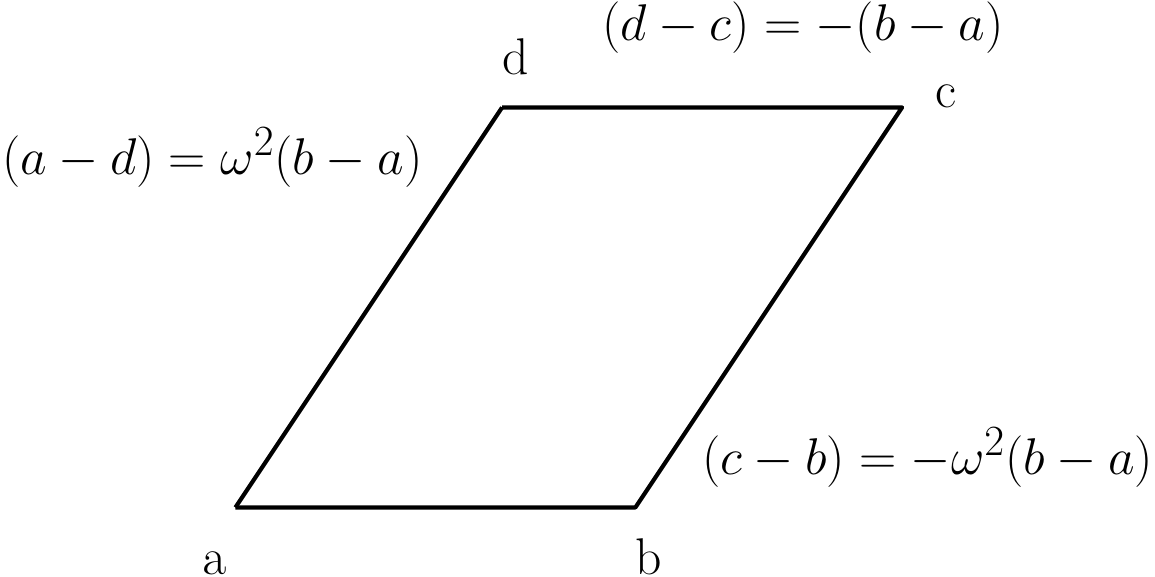}
\caption{A unit rhombus whose vertices occupy positions $a, b, c$ and $d$ in the complex plane. Here $z = b-a$.}
\end{center}
\label{fig:aminusd}
\end{figure}

Let $a, b, c$ and $d$ be the vertices of a lattice rhombus of $ \mathbb{L}$,  of side $1$ as in Figure~\ref{fig:aminusd} such that 
 \beq\lab{eq:1.3}  a - d = z\omega^2, \eeq \beq b-a = z,\eeq \beq c-b = -z \omega^2 ,\eeq  \beq\lab{eq:1.6}  d-c = -z,\eeq for some $z \in \{1, \omega, \omega^2\}.$ In the respective cases when $z= 1, \omega$ or $\omega^2$, we define corresponding sets of lattice rhombi of side $1$ to be $E_0(\mathbb L)$, $E_1(\mathbb L)$ or $E_2(\mathbb L)$. Note that $a$ and $c$ are vertices at which the angle is $\frac{\pi}{3}$. 
For $i = 0, 1$ and $2$, we define $E_i(\T_n)$ analogously.
For $s_0, s_1, s_2 > 0$ and $f:V(\T_n) \ra \R$, we say that $g = \nabla^2(f)$ satisfies $g \preccurlyeq s = (s_0, s_1, s_2)$, if  for all $a, b, c$ and $d$ satisfying (\ref{eq:1.3}) to (\ref{eq:1.6}) and $e = \{a, b, c, d\}$, $g$ satisfies

\ben 
\item  $g(e) \leq  s_0,$ if $e \in E_0(\T_n)$, i.e. $z = 1$.
\item $g(e)  \leq  s_1,$  if $e \in E_1(\T_n)$, i.e.  $z = \omega$.
\item $g(e) \leq  s_2,$ if  $e \in E_2(\T_n)$ i.e. $z = \omega^2$.
\een

In the respective cases when $z= 1, \omega$ or $\omega^2$, we define corresponding sets of lattice rhombi of side $1$ to be $E_0(\mathbb L)$, $E_1(\mathbb L)$ or $E_2(\mathbb L)$. This structure is carried over to $\T_n$ by the map $\phi_{0, n}$ defined in the beginning of Subsection~\ref{sec:cover}.
In the beginning of Subsection~\ref{sec:cover}, we have mapped $V(\T_n)$ on to $(\Z/n\Z) \times (\Z/n\Z)$ by mapping $1$ to $(1, 0)$ and $\omega$ to $(0, 1)$ and extending this map to $V(\T_n)$ via a $\Z$ module homomorphism. In particular, this maps $1 + \omega$ to $(1, 1)$.

We will further assume that $2 = s_0 \leq s_1 \leq s_2.$  Given  $s = (s_0, s_1, s_2)\in \R_+^3,$  let $P_n(s)$ be the bounded polytope of  all functions $g:V(\T_n) \ra \R$ such that $\sum_{v \in V(\T_n)} g(v) = 0$ and $\nabla^2(g)\preccurlyeq s$.

\begin{defn}Let $\tP_n(s)$ be defined to be the following image of $P_n(s)$ under an affine transformation.  Given  $s = (s_0, s_1, s_2)\in \R_+^3,$  let $\tP_n(s)$ be the bounded polytope of  all functions $g:V(\T_n) \ra \R$ such that $g(0) = 0$ and $\nabla^2(g)\preccurlyeq s$. 
\end{defn}
We observe that the $n^2-1$ dimensional Lebesgue measures of $\tP_n(s)$ and  $P_n(s)$ satisfy
$$|\tP_n(s)|^{1/n^2}\left(1 - \frac{C\log n}{n}\right) \leq |P_n(s)|^{1/n^2} \leq |\tP_n(s)|^{1/n^2}\left(1 + \frac{C\log n}{n}\right).$$

\begin{lem}\lab{lem:2.3}
For any  $s = (s_0, s_1, s_2)$, where $2 = s_0 \leq s_1 \leq s_2$, there is a unique quadratic function $q(s)$ from $\mathbb L$ to $\R$ such that $\nabla^2q$ satisfies the following.

\ben 
\item  $\nabla^2q(e) = - s_0,$ if $e \in E_0(\mathbb L)$.
\item $\nabla^2q(e)  =  - s_1,$  if $e \in E_1(\mathbb L )$.
\item $\nabla^2q(e)  =  - s_2,$ if  $e \in E_2(\mathbb L)$.
\item $q(0) = q(n) = q(n\omega) = 0$.
\een
\end{lem}
\begin{proof}
This can be seen by explicitly constructing $q(s)$ when $s = (1, 0, 0)$, $(0, 1, 0)$ and $(0,0,1)$ (which are rotations of the same concave function) and combining these by linear combination. 
\end{proof}

Given a concave function $f:\mathbb L \ra \R$ such that 
$\nabla^2 f$ is invariant under translation by elements of $n \mathbb L$, and the average value of $\nabla^2 f$ on edges  in $E_i(\L)$ (which is well defined due to periodicity) is equal to $-s_i$ ,  and $f(0) = f(n) = f(n\omega) = 0$, we consider $(f - q)(s)$. Since the average value of $\nabla^2 f - \nabla^2 q$ is $0$, this implies that $f - q$ is $0$ on $n\L$, and more generally, is invariant under translations in $n\L$. We can therefore view $f - q$ to be a function from $\T_n = \L/{n\L}$ to $\R$, 
and in fact the resulting function is in $\tP_n(s)$. Conversely, any point in $\tP_n(s)$ can be extended to a periodic function on $\L$, to which we can add $q(s)$ and thereby recover a function $f$ on $\L$ that is concave, such that
$\nabla^2 f$ is invariant under translation by elements of $n \mathbb L$, the average value of $\nabla^2 f$ on $E_i(\L)$ is $-s_i$ ,  and $f(0) = f(n) = f(n\omega) = 0$.

\underline{Note on constants:} We will denote  constants depending only on $s$ by $C$ and $c$

\subsection{Convex geometry}
Let $1 \leq \ell \in \Z$. Given sets $K_i\subseteq \R^m$ for $i \in [\ell]$, let their Minkowski sum $\{x_1 + \dots + x_\ell \big| \forall i \in [\ell], x_i \in K_i\},$ be denoted by $K_1 + \dots + K_\ell.$

Let $K$ and $L$ be compact convex subsets of $\R^m$. 

Then, the Brunn-Minkowski inequality \cite{Brunn, Minkowski} states that \beq |K + L|^\frac{1}{m} \geq |K|^\frac{1}{m} + |L|^\frac{1}{m}.\eeq 
It can be shown that 
$$\lim_{\eps \ra 0^+} \frac{|L + \eps K| - |L|}{\eps}$$ exists. We will call this the anisotropic surface area $S_K(L)$ of $L$ with respect to $K$.

 Dinghas \cite{Dinghas, Figalli} showed that the following anisotropic isoperimetric inequality can be derived from the Brunn-Minkowski inequality.
\beq\lab{eq:2.2} S_K(L) \geq m|K|^{\frac{1}{m}} |L|^{\frac{m-1}{m}}.\eeq 

We shall need the following result of Pr\'{e}kopa (\cite{prekopa}, Theorem 6).
\begin{thm}\lab{thm:prekopa}
Let $f(x, y)$ be a function of $\R^n \oplus \R^m$ where $x \in \R^n$ and 
and $y \in \R^m$. Suppose that $f$ is logconcave in $\R^{n+m}$ and let
$A$ be a convex subset of $\R^m$. Then the function of the variable x:
$$\int_A f(x, y) dy$$
is logconcave in the entire space $\R^n$.
\end{thm}

We note the following theorem of Fradelizi \cite{Fradelizi}. 
\begin{thm}\lab{thm:frad}
 The density at the center of mass of a logconcave density on $\R^{n}$ is no less than $e^{- n}$ multiplied by the supremum of the density.
\end{thm}

We will also need the following theorem of Vaaler \cite{Vaaler}.
\begin{thm}\lab{thm:vaal}
There is a  lower bound of $1$ on the volume of a central section of the unit  cube.
\end{thm}



\section{Characteristics of relevant polytopes}\lab{sec:2}

\subsection{Volume of the polytope $P_n(s)$}

We denote the $k-$dimensional Lebesgue measure of a $k-$dimensional polytope $P$ by $|P|$. We will need to show that $|P_m(s)|^{1/m^2}$ is less than $(1 + o_m(1))|P_n(s)|^{\frac{1}{n^2}},$ for $ n \geq m$. We achieve this by conditioning on a ``double layer boundary" and the use of the Brunn-Minkowski inequality.
We will identify $\Z + \Z\omega$ with $\Z^2$ by mapping $x + \omega y$, for $x, y \in \Z$ onto $(x, y) \in \Z^2.$

Given $n_1|n_2$, the natural map from $\Z^2$ to $\Z^2/(n_1 \Z^2) = \T_{n_1}$ factors through $\Z^2/(n_2 \Z^2) =\T_{n_2}$. We denote the respective resulting maps from $\T_{n_2}$ to $\T_{n_1}$ by $\phi_{n_2, n_1}$, from $\Z^2$ to $\T_{n_2}$ by $\phi_{0, n_2}$ and from $\Z^2$ to $\T_{n_1}$ by $\phi_{0, n_1}$.
Given a set of boundary nodes $\b \subseteq V(\T_n)$, and $ x \in \R^{\b}$, we define $Q_{ \b}(x)$ to be the fiber polytope over $x$, that arises from the projection map $\Pi_{\b}$ of $\tP_n(s)$ onto $\R^{\b}.$ Note that $Q_{\b}(x)$ implicitly depends on $s$.


\begin{lem}\lab{lem:3-}
Let $\{0\} \subseteq \b_1 \neq \{0\},$ be a subset of $V(\T_{n_1}).$ Then,
$$0 \leq \ln |\Pi_{\b_1} \tP_{n_1}(s)| \leq (|\b_1| -1)\ln (Cn_1^2).$$
\end{lem}
\begin{proof}
 Given any vertex $v_1$ in $\b_1$ other than $0$, there is a lattice path $\mathrm{path}(v_1)$ (i.e. a path $0= a_1, \dots, a_k = v_0$, where each $a_i - a_{i-1}$ is in the set $\{1, 1 + \omega, \omega, -1, \omega^2, 1 - \omega^2\}$) that goes from $0$ to some vertex $v_0 \in \phi^{-1}_{0,n_1}(v_1)$ that consists of two straight line segments, the first being from $0$ to some point in $\Z^+$, and the second having the direction $1 + \omega$. It is clear that this $v_0$ can be chosen to have absolute value at most $2n_1$ by taking an appropriate representative of $\phi^{-1}_{0,n_1}(v_1).$
We see that $[0,1]^{\b_1\setminus\{0\}} \subseteq \Pi_{\b_1} \tP_{n_1}(s) \subseteq \R^{\b_1\setminus \{0\}}.$ Let $f_1 \in  \tP_{n_1}(s)$. Along $\mathrm{path}(v_1)$, at each step, the slope of $f$ increases by no more than a constant, due to the condition $\nabla^2(f_1) \preccurlyeq s.$ This implies that $f_1$ is $Cn_1$ Lipschitz. Therefore, $\|f_1\|_{\ell_\infty}$ is at most $Cn_1^2.$ Thus $\Pi_{\b_1} \tP_{n_1}(s)$ is contained inside a $|\b_1| -1$ dimensional cube of side length no more than $Cn_1^2.$ 
We have thus proved the lemma.
\end{proof}
\begin{lem}\lab{lem:3}
Let $n_1$ and $n_2$ be positive integers satisfying $n_1 | n_2$. Then 
\beq 1 \leq |\tP_{n_1}(s)|^{\frac{1}{n_1^2}} \leq |\tP_{n_2}(s)|^{\frac{1}{n_2^2}}\left(1 + \frac{C\log n_1}{n_1}\right).\eeq 
\end{lem}
\begin{proof}The lower bound of $1$ on $|\tP_{n_1}(s)|^{\frac{1}{n_1^2}}$ follows from $[0,1]^{V(\T_{n_1})\setminus\{0\}} \subseteq \tP_n(s).$
We define the set $\b_{1} \subseteq V(\T_{n_1})$ of ``boundary vertices" to be all vertices that are either of the form $(0, y)$ or $(1, y)$ or $(x, 0)$ or $(x, 1)$, where $x, y$ range over all of $\Z/(n_1 \Z)$. We define the set $\b_{2}$ to be $\phi_{n_2, n_1}^{-1}(\b_{1}).$ For $x \in \R^{\b_1}$, let $F_1(x):= |Q_{\b_1}(x)|,$ and for  $x \in \R^{\b_2}$, let $F_2(x):= |Q_{\b_2}(x)|.$
We now have \beq|\tP_{n_1}(s)| = \int\limits_{\R^{\b_1}}F_1(x)dx = \int\limits_{\Pi_{\b_1} \tP_{n_1}(s)} F_1(x) dx.\eeq
Let  $\phi^*_{n_2, n_1}$ be the linear map from $\R^{V(\T_{n_1})}$ to $\R^{V(\T_{n_2})}$ induced by $\phi_{n_2, n_1}$. 
Let  $\psi_{\b_1, \b_2}$ be the linear map from $\R^{\b_1}$ to $\R^{\b_2}$ induced by $\phi_{n_2, n_1}$. 
Then, for $x \in \R^{\b_1},$ \beq F_2(\psi_{\b_1, \b_2}(x)) = F_1(x)^{\left(\frac{n_2}{n_1}\right)^2}.\lab{eq:2.6}\eeq
Note that that $\tP_{n}(s)$ is $n^2-1$ dimensional, has an $\ell_\infty$ diameter of $O(n^2)$ and contains a $n^2-1$ dimensional unit $\ell_\infty-$ball as a consequence of $s_0$ being set to  $2$. So the $|\b_1|-1$ dimensional polytopes  $\Pi_{\b_1} \tP_{n_1}(s)$, and $\psi_{\b_1, \b_2}(\Pi_{\b_1} \tP_{n_1}(s))$ contain $|\b_1|-1$ dimensional $\ell_\infty$ balls of radius $1.$ 
\begin{claim} \lab{cl:2.2}
Let $S_{\b_1, \b_2}(\frac{1}{n_1^{4}})$ be the set of all $y \in \R^{\b_2}$ such that there exists $x \in  \Pi_{\b_1} \tP_{n_1}((1 - \frac{1}{n_1^2})s)$ for which  $ y - \psi_{\b_1, \b_2}(x) \perp \psi_{\b_1, \b_2}(\R^{\b_1})$  and $\|y - \psi_{\b_1, \b_2}(x) \|_{\ell_\infty} < \frac{1}{n_1^{4}}.$ Then, $y \in S_{\b_1, \b_2}(\frac{1}{n_1^{4}})$ implies the following. \ben \item $y \in \Pi_{\b_2} \tP_{n_2}((1 - \frac{1}{2n_1^2})s)$ and \item 
 $|Q_{\b_2} (y)| \geq c^{(\frac{n_2}{n_1})^2} |Q_{\b_2} (\psi_{\b_1, \b_2}(x))|.$\een
\end{claim}
\begin{proof}
The first assertion of the claim follows from the triangle inequality. To see the second assertion, 
let the vector $w \in \R^{V(\T_{n_2})}$ equal $0$ on all the coordinates indexed by $V(\T_{n_2})\setminus \b_2$ and equal $\psi_{\b_1, \b_2}(x) - y$ on coordinates indexed by $\b_2$.
We know that $x \in  \Pi_{\b_1} \tP_{n_1}((1 - \frac{1}{n_1^2})s)$. Therefore, \ben \item[$(\ast)$] $Q_{\b_2} (\psi_{\b_1, \b_2}(x)) -w$ has dimension $n_2^2 - |\b_2|$, and contains an axis aligned cube of side length $\frac{c}{n_1^2},$ and hence a euclidean ball of radius $\frac{c}{n_1^2}$.\een Since every constraint defining $\tP_{n_2}(s)$ has the form $x_a + x_b - x_c - x_d \leq s_i,$ or $x_0 = 0$, \ben \item[$(\ast \ast)$] the affine spans of the codimension $1$ faces of the fiber polytope $Q_{\b_2}(y)$ are respectively translates of the affine spans of the corresponding codimension $1$ faces of $Q_{\b_2} (\psi_{\b_1, \b_2}(x)) -w$ by  euclidean distances that do not exceed $\frac{C}{ n_1^{4}}.$ \een
Therefore, by $(\ast)$ and $(\ast \ast)$, some translate of $(1 - \frac{C}{n_1^{2}})Q_{\b_2} (\psi_{\b_1, \b_2}(x))$ is contained inside $Q_{\b_2} (y)$, completing the proof of Claim~\ref{cl:2.2}.
\end{proof}
Let $K$ denote the intersection of the origin symmetric cube of radius $\frac{1}{n_1^{4}}$ in $\R^{\b_2}$ with the orthocomplement of $\psi_{\b_1, \b_2}(\R^{\b_1})$. By the lower bound of $1$ on the volume of a central section of the unit  cube (due to Vaaler \cite{Vaaler}), it follows that the volume of $K$ is at least $\left(\frac{1}{n_1^{4}}\right)^{|\b_2| - |\b_1|}.$
The inequalities below now follow from (\ref{eq:2.6}) and Claim~\ref{cl:2.2}.
\beqs |\tP_{n_2}(s)|  & = & \int\limits_{\Pi_{\b_2} \tP_{n_2}(s)}|Q_{\b_2}(y)|dy\\
                             & \geq & \int\limits_{\Pi_{\b_2} \tP_{n_2}((1 - \frac{1}{2n_1^2})s)} F_2(y) dy\lab{eq:2.8}\\
& \geq & \int\limits_{ S_{\b_1, \b_2}(\frac{1}{n_1^4})} F_2(y) dy\\
& \geq & \vol(K)\int\limits_{\psi_{\b_1, \b_2}(\Pi_{\b_1} \tP_{n_1}((1 - \frac{1}{n_1^2})s))} c^{(\frac{n_2}{n_1})^2}F_2(z) dz\\
& \geq &  \vol(K)\int\limits_{\Pi_{\b_1} \tP_{n_1}((1 - \frac{1}{n_1^2})s)}c^{(\frac{n_2}{n_1})^2} F_1(x)^{\left(\frac{n_2}{n_1}\right)^2}dx\nonumber\\
& \geq &  c^{(\frac{n_2}{n_1})^2}\left(\frac{1}{n_1^{4}}\right)^{|\b_2|-|\b_1|}\int\limits_{\Pi_{\b_1} \tP_{n_1}((1 - \frac{1}{n_1^2})s)} F_1(x)^{\left(\frac{n_2}{n_1}\right)^2}dx.\eeqs

By Lemma~\ref{lem:3-}, $ n_1^{-Cn_1} \leq |\Pi_{\b_1} \tP_{n_1}(s)| \leq n_1^{Cn_1}$, for some universal positive constant $C > 1.$ Also, $c  |\Pi_{\b_1} \tP_{n_1}(s)| \leq |\Pi_{\b_1} \tP_{n_1}((1 - \frac{1}{n_1^2})s)| \leq |\Pi_{\b_1} \tP_{n_1}(s)|.$
\beqs \int\limits_{\Pi_{\b_1} \tP_{n_1}((1 - \frac{1}{n_1^2})s)} F_1(x)^{\left(\frac{n_2}{n_1}\right)^2}dx & \geq & |\Pi_{\b_1} \tP_{n_1}((1 - \frac{1}{n_1^2})s)|^{1 - (n_2/n_1)^2}  \\ &\times&\left(\int\limits_{\Pi_{\b_1} \tP_{n_1}((1 - \frac{1}{n_1^2})s)} F_1(x)dx\right)^{\left(\frac{n_2}{n_1}\right)^2}\\
& \geq &  |\Pi_{\b_1} \tP_{n_1}(s)|^{1 - (n_2/n_1)^2}  | \tP_{n_1}((1 - \frac{1}{n_1^2})s)|^{\left(\frac{n_2}{n_1}\right)^2}\\
& \geq &  |\Pi_{\b_1} \tP_{n_1}(s)|^{1 - (n_2/n_1)^2}  \left(c| \tP_{n_1}(s)|\right)^{\left(\frac{n_2}{n_1}\right)^2}\\
& \geq & (Cn_1^{Cn_1})^{1 - (n_2/n_1)^2} | \tP_{n_1}(s)|^{\left(\frac{n_2}{n_1}\right)^2}. \eeqs
Thus,
\beq | \tP_{n_1}(s)|^{\left(\frac{n_2}{n_1}\right)^2} \leq (Cn_1^{Cn_1})^{(n_2/n_1)^2-1} \left(n_1^{4}\right)^{|\b_2|-|\b_1|}|\tP_{n_2}(s)|,  \eeq
which gives us \beq| \tP_{n_1}(s)|^{\left(\frac{1}{n_1}\right)^2} & \leq & (Cn_1^{Cn_1})^{(1/n_1^2)-(1/n_2^2)} \left(n_1^{4}\right)^{\frac{|\b_2|-|\b_1|}{n_2^2}}|\tP_{n_2}(s)|^{\frac{1}{n_2^2}}\\
& \leq & |\tP_{n_2}(s)|^{\frac{1}{n_2^2}}n_1^{\frac{C}{n_1}}\\
& \leq & |\tP_{n_2}(s)|^{\frac{1}{n_2^2}}\left(1 + \frac{C \log n_1}{n_1}\right).\eeq
\end{proof}
For a positive integer $n$, let $[n]$ denote the set of positive integers less or equal to $n$, and let $[n]^2$ denote $[n]\times [n]$.
In what follows, we will use $v$ to denote an arbitrary vertex in $V(\T_{n_3})$. Then, by symmetry, \beq \frac{\int_{P_{n_3}(s)} x(v) dx}{|P_{n_3}(s)|} & = &  \left(\frac{1}{n_3^2}\right)\sum_{v' \in V(\T_{n_3})}  \frac{\int_{P_{n_3}(s)} x(v') dx}{ |P_{n_3}(s)|} \\
& = &   \frac{\int_{P_{n_3}(s)} \left(\frac{\sum_{v' \in V(\T_{n_3})} x(v')}{n_3^2}\right)dx}{|P_{n_3}(s)|} \\
& = & 0.\lab{eq:2.10} \eeq
The linear map $u:P_{n_3}(s) \rightarrow \tP_{n_3}(s)$ defined  by $u(x)(v) = x(v) - x(0)$ is surjective and volume preserving. Therefore,  
\beq \frac{\int_{\tP_{n_3}(s)} x(v) dx}{|\tP_{n_3}(s)|} & = &   \frac{\int_{P_{n_3}(s)} u(x)(v) dx}{ |P_{n_3}(s)|} \\
& = &  \frac{\int_{P_{n_3}(s)} x(v) dx}{|P_{n_3}(s)|} -  \frac{\int_{P_{n_3}(s)} x(0) dx}{|P_{n_3}(s)|}\\
& = & 0.\eeq
\begin{lem}\lab{lem:4}
Let $C < n_2 < n_3$. Then, 
\beq |P_{n_2}(s)|^{\frac{1}{n_2^2}} \geq |P_{n_3}(s)|^{\frac{1}{n_3^2}}\left(1 - \frac{C (n_3 - n_2) \ln n_3}{n_3}\right).\eeq  
\end{lem}
\begin{proof}  Let $\rho:V(\T_{n_2}) \ra [n_2]^2\subseteq \Z^2$ be the unique map that satisfies $\phi_{0, n_2} \circ \rho = id$ on $V(\T_{n_2})$. We embed $V(\T_{n_2})$ into $V(\T_{n_3})$ via  
  $\phi_{0, n_3} \circ \rho,$ and define $\b$ to be $ V(\T_{n_3})\setminus (\phi_{0, n_3} \circ \rho(V(\T_{n_2}))).$  Note that $0 \in \b$, since $0 \not \in [n_2].$ Recall that $Q_{ \b}(x)$ was defined to be the fiber polytope over $x$, that arises from the projection map $\Pi_{\b}$ of $\tP_n(s)$ onto $\R^{\b}.$
Thus,
\beqs \int\limits_{\R^{\b\setminus \{0\}}} \left(\frac{|Q_{\b}(x)| }{|\tP_{n_3}(s)|}\right)x dx & = & \Pi_\b  \left(\frac{\int_{\tP_{n_3}(s)} x(v) dx}{|\tP_{n_3}(s)|}\right)\\ 
& = & 0.\eeqs
By Theorem~\ref{thm:prekopa}, $\frac{|Q_{\b}(x)| }{|\tP_{n_3}(s)|}$ is a logconcave function of $x\in \tP_{n_3}(s)$. 
$\frac{|Q_{\b}(x)| }{|\tP_{n_3}(s)|}$ is a non-negative and integrable function of $x$, and hence by the Brunn-Minkowski inequality, it follows that
\beqs \int\limits_{\R^{\b\setminus \{0\}}} \left(\frac{|Q_{\b}(x)|}{|\tP_{n_3}(s)|}\right)|Q_{\b}(x)|^{\frac{1}{n_3^2 - |\b|}} dx \leq |Q_{\b}(0)|^{\frac{1}{n_3^2 - |\b|}} .\eeqs
Therefore, 
\beqs \int\limits_{\Pi_{\b} \tP_{n_3}(s)} |Q_{\b}(x)|^{1 + \frac{1}{n_3^2 - |\b|}} \left(\frac{dx}{|\Pi_\b\tP_{n_3}(s)|}\right) \leq \left(\frac{| \tP_{n_3}(s)|}{|\Pi_{\b} \tP_{n_3}(s)|}\right)|Q_{\b}(0)|^{\frac{1}{n_3^2 - |\b|}}.\eeqs
By the monotonic increase of $L_p(\mu)$ norms as $p$ increases from $1$ to $\infty$,  for the probability measure $\mu(dx) = \frac{dx}{|\Pi_\b\tP_{n_3}(s)|}$, we see that 
\beq \int\limits_{\Pi_{\b}\tP_{n_3}(s)} |Q_{\b}(x)|^{1 + \frac{1}{n_3^2 - |\b|}} \frac{dx}{|\Pi_\b \tP_{n_3}(s)|} & \geq &  \left(\int\limits_{\Pi_{\b}\tP_{n_3}(s)} |Q_{\b}(x)| \frac{dx}{|\Pi_\b \tP_{n_3}(s)|}\right)^{1 + \frac{1}{n_3^2 - |\b|}}\\ & = & \left(\frac{| \tP_{n_3}(s)|}{|\Pi_{\b} \tP_{n_3}(s)|}\right)^{1 + \frac{1}{n_3^2 - |\b|}}.\eeq
It follows that \beq |Q_{\b}(0)| \geq \frac{| \tP_{n_3}(s)|}{|\Pi_{\b} \tP_{n_3}(s)|}.\eeq
Suppose that $n_2 + 2 < n_3$.
Let $\rho_+:V(\T_{n_2+2}) \ra [n_2+2]^2\subseteq \Z^2$ be the unique map that satisfies $\phi_{0, n_2+2} \circ \rho_+ = id$ on $V(\T_{n_2+2})$. We embed $V(\T_{n_2+2})$ into $V(\T_{n_3})$ via  
  $\phi_{0, n_3} \circ \rho_+,$ and define $\tb$ to be $ V(\T_{n_3})\setminus (\phi_{0, n_3} \circ \rho_+(V(\T_{n_2+2}))).$ 
We observe that $|\tP_{n_2+2}(s(1 + \frac{2}{(n_2+2)^2}))|$ is greater or equal to $|Q_{\b}(0)|(\frac{1}{(n_2+2)^2}))^{|\b|-|\tb|},$ since $\phi_{0, n_3} \circ \rho_+,$ induces an isometric map from $Q_\b(0) + [0, \frac{1}{(n_2+2)^2}]^{\b\setminus\tb}$ into $\tP_{n_2+2}(s(1 + \frac{2}{(n_2+2)^2})).$
Thus,
\beqs |\tP_{n_2 + 2}(s)| & = &  (1 + \frac{2}{(n_2+2)^2})^{-(n_2+2)^2+1}|\tP_{n_2+2}(s(1 + \frac{2}{(n_2+2)^2}))|\\
&  \geq & e^{-2} |Q_{\b}(0)|(\frac{1}{(n_2+2)^2})^{|\b|-|\tb|}\\
&  \geq & \frac{e^{-2}| \tP_{n_3}(s)|(\frac{1}{(n_2+2)^2})^{|\b|-|\tb|} }{|\Pi_{\b} \tP_{n_3}(s)|}\\
& \geq & | \tP_{n_3}(s)| (Cn_3)^{-Cn_3(n_3-n_2)}.\eeqs
Noting that $\tP_{n_2+2}(s)$ contains a unit cube and hence has volume at least $1$, we see that 
\beq |\tP_{n_2 + 2}(s)|^{\frac{1}{(n_2 + 2)^2}} & \geq & |\tP_{n_2+2}(s)|^{\frac{1}{n_3^2}}\\
& \geq & | \tP_{n_3}(s)|^{\frac{1}{n_3^2}} (C n_3)^{-C(1-\frac{n_2}{n_3})}\\
& \geq &  | \tP_{n_3}(s)|^{\frac{1}{n_3^2}} \left(1 - \frac{C (n_3 - n_2) \ln n_3}{n_3}\right).\eeq
Noting that $n_2 + 2 < n_3$ and relabeling $n_2 + 2$ by $n_2$  gives us the lemma.
\end{proof}


We will need the notion of differential entropy (see page 243 of \cite{Cover}).
\begin{defn}[Differential entropy]
Let ${\displaystyle X}$ be a random variable supported on a finite dimensional Euclidean space $\R^m$,
associated with a measure $\mu$ that is absolutely continuous with respect to the Lebesgue measure. Let the Radon-Nikodym derivative of $\mu$ with respect to the Lebesgue measure be denoted $f$. The differential entropy of $X$, denoted ${\displaystyle h(X)}$ (which by overload of notation, we shall also refer to as the differential entropy of $f$, i.e. $h(f)$), is defined  as
${\displaystyle h(X)=-\int _{\R^m}f(x)\ln f(x)\,dx}$.
\end{defn}
The following Lemma is well known, but we include a proof for the reader's convenience.
\begin{lem}\lab{lem:5}
The differential entropy of a  mean $1$ distribution with a bounded Radon-Nikodym derivative with respect to the Lebesgue measure, supported on $[0, \infty)$ is less or equal to $1$, and equality is achieved on the exponential distribution.
\end{lem}
\begin{proof}
Let $f:[0, \infty) \ra \R$ denote a density supported on the non-negative reals, whose associated distribution $F$ has mean $1$. Let $g:[0, \infty) \ra \R$ be given by $g(x) := e^{-x}$. The relative entropy between $f$ and $g$ is given by 

\beq D(f||g) := \int_{[0, \infty)} f(x) \ln\left(\frac{f(x)}{g(x)}\right)dx,\eeq and can be shown to be non-negative for all densities $f$ using Jensen's inequality.
We observe that \beq D(f||g) & = & -h(f) +  \int_{[0, \infty)} f(x) \ln\left({e^{x}}\right)dx\\
                                          & = & - h(f) + 1, \eeq because $F$ has mean $1$.
This implies that $h(f) \leq 1 = h(g)$.
\end{proof}

\begin{figure}\label{fig:factorize}
\begin{center}
\includegraphics[scale=0.75]{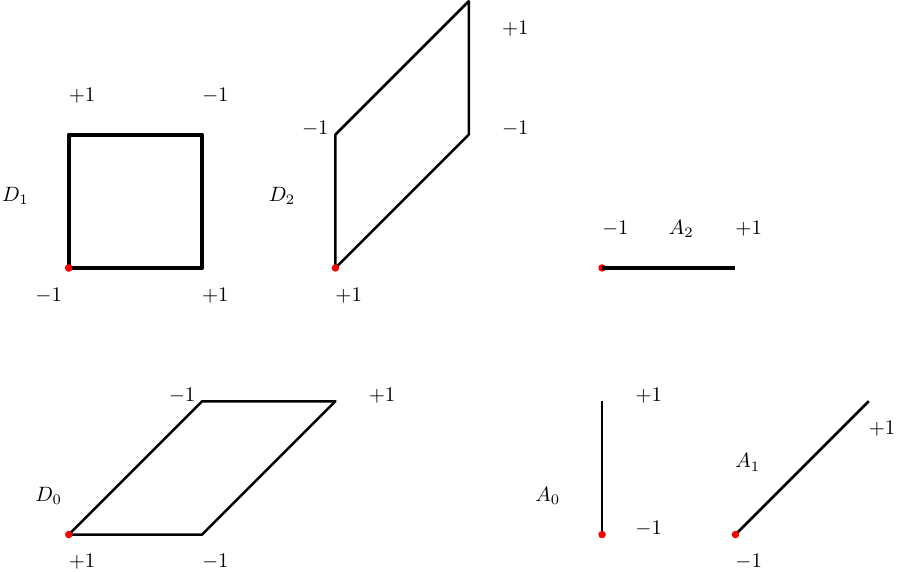}
\caption{We factorize the second order $D_i$ into first order operators $A_i$. A red dot indicates the point at which the operator is evaluated.}
\end{center}
\end{figure}

We define the first order difference operators $A_0$, $A_1$ and $A_2$ on $\R^{V(\T_n)}$ given by 
\beq\lab{eq:A}
A_0 f(v_1-1, v_2-1) =  - f(v_1-1, v_2-1) + f(v_1 - 1, v_2).\nonumber\\
A_1 f(v_1, v_2 ) =   - f(v_1 - 1, v_2 - 1) + f(v_1, v_2).\nonumber\\
A_2 f(v_1-1, v_2-1) =  - f(v_1 - 1, v_2 - 1) + f(v_1, v_2).\nonumber\\
\eeq

\begin{lem}\lab{lem:6}
If $2 = s_0 \leq s_1 \leq s_2$, \beqs |P_{n}(s)| \leq \exp\left((1 + \ln (2(1 + C/n))) n^2\right).\eeqs
\end{lem}
\begin{proof}
The map $\phi$ that takes $x \in P_{n}(s)$ to $\phi(x) \in \R^{[n]^2}$, where  $(\phi(x))_{(p_i, q_i)}$ equals $x_{(p_i, q_i)} - x_{(p_{i}, q_{i-1})}  - x_{(p_{i-1}, q_{i})} + x_{(p_{i-1}, q_{i-1})}$ is a linear transformation. Further, the image contains a codimension 1 section of an cube of sidelength $1$, which by Vaalar's Theorem has volume at least 1. The  Jacobian determinant of the transformation matrix from the set of points in $\R^{[n^2]}$, the sum of whose coordinates is $0$, to itself  has magnitude at least $1$ since the matrix is integral. Thus it suffices to bound from above,  $|\phi(P_{n}(s))|.$ Let $x'$ be sampled uniformly at random from $\phi(P_{n}(s))$. We also add to $x'$ an independent random vector $y'$ from the one dimensional line segment centered at $0$, perpendicular to $P_n(s)$ of length $1$. We then see that $x' + y'$ has mean $0$, and further, each coordinate is bounded above by $2(1 + C/n)$. Therefore, by Lemma~\ref{lem:5}, the differential entropy of each coordinate is at most $1 + \ln (2(1 + C/n))$. Since it is well known that the joint differential entropy of a vector valued random variable, is less or equal to the sum of the differential entropies of its marginals, we see that 
$$|\phi(P_{n}(s))|  \leq (2(1+ C/n) e)^{n^2}.$$ This proves the Lemma.
\end{proof}

We will use the lemmas in this section to prove the following.

\begin{lem}\lab{lem:2.8}
Let $s_0 = 2$. Then, as $n\ra \infty$, $|P_n(s)|^{\frac{1}{n^2}}$  converges to a limit in the interval $[1, 2e]$.
\end{lem}
\begin{proof}
By Lemma~\ref{lem:3} and Lemma~\ref{lem:6}, \beq 1 \leq  \liminf\limits_{n \ra \infty} |P_n(s)|^{\frac{1}{n^2}} \leq \limsup\limits_{n \ra \infty} |P_n(s)|^{\frac{1}{n^2}} \leq 2e.\eeq

Let $C < n_1^2  \leq n_2.$ Let $n_3 = (\lfloor \frac{n_2}{n_1}\rfloor + 1)  n_1.$ 
By Lemma~\ref{lem:3} and Lemma~\ref{lem:4},
\beqs |P_{n_1}(s)|^{\frac{1}{n_1^2}} & \leq &  |P_{n_3}(s)|^{\frac{1}{n_3^2}}\left(1 + \frac{C\log n_1}{n_1}\right)\\
& \leq & |P_{n_2}(s)|^{\frac{1}{n_2^2}}\left(1 - \frac{C (n_3 - n_2) \ln n_3}{n_3}\right)^{-1}\left(1 + \frac{C\log n_1}{n_1}\right)\\
& \leq & |P_{n_2}(s)|^{\frac{1}{n_2^2}}\left(1 - \frac{C n_1 \ln n_3}{n_3}\right)^{-1}\left(1 + \frac{C\log n_1}{n_1}\right)\\ 
& \leq & |P_{n_2}(s)|^{\frac{1}{n_2^2}}\left(1 - Cn_1 \left(\frac{\ln n_1^2}{n_1^2}\right)\right)^{-1}\left(1 + \frac{C\log n_1}{n_1}\right).\eeqs
This implies that \beqs |P_{n_2}(s)|^{\frac{1}{n_2^2}} \geq |P_{n_1}(s)|^{\frac{1}{n_1^2}}\left(1 - \frac{C\log n_1}{n_1}\right).\eeqs
As a consequence, \beqs \left(1 + \frac{C\log n_1}{n_1}\right)\liminf\limits_{n_2 \ra \infty} |P_{n_2}(s)|^{\frac{1}{n_2^2}} \geq |P_{n_1}(s)|^{\frac{1}{n_1^2}}. \eeqs Finally, this gives 
\beqs \liminf\limits_{n_2 \ra \infty} |P_{n_2}(s)|^{\frac{1}{n_2^2}} \geq \limsup\limits_{n_1 \ra \infty}|P_{n_1}(s)|^{\frac{1}{n_1^2}}, \eeqs implying 
\beqs 1 \leq  \liminf\limits_{n \ra \infty} |P_n(s)|^{\frac{1}{n^2}} = \lim\limits_{n \ra \infty} |P_n(s)|^{\frac{1}{n^2}} = \limsup\limits_{n \ra \infty} |P_n(s)|^{\frac{1}{n^2}} \leq 2e.\eeqs
\end{proof}
Together with the concavity of $\f_n:= |P_n(s)|^{\frac{1}{n^2-1}}$, this implies the following.
\begin{cor}\lab{cor:lip}
Let $\eps > 0$. For all sufficiently large $n$, for all $s$ and $t$ in $\R_+^3$, \beqs |\f_n(s) - \f_n(t)| < \sqrt{2}(2e + \eps)|s - t|.\eeqs
\end{cor}
\begin{proof}
For $u$ such that $s - u \in \R^3_+$, we know that \beqs |\f_n(s) -\f_n(u)| <|\f_n(s - u)| \leq (2e + \eps)|s - t|.\eeqs
Consider the line through $s$ and $t$. We introduce $u = (\min(s_0, t_0), \min(s_1, t_1), \min(s_2, t_2))$, and note that 
$$|\f_n(s) - \f_n(u) - (\f_n(t )- \f_n(u))|< \max(\f_n(s) - \f_n(u), \f_n(t )- \f_n(u))$$ because $s - u$ and $t - u$ belong to $\R_+^3$. Noting that $(s - u)\cdot (t - u) \geq 0$, we have $$\max(|s -u|, |t - u|) \leq \sqrt{2} |s - t|.$$ The corollary follows by the concavity of $f_n$  on the intersection of this line with $\R_+^3$, the fact that $\f_n$ tends to $0$ on the boundary of $\R_+^3$, Lemma~\ref{lem:3} and Lemma~\ref{lem:2.8} .
\end{proof}
\begin{cor}
The pointwise limit of the functions $\mathbf{f}_n$ is a function $\mathbf{f}$ that is  $2\sqrt{2}e$ Lipschitz and concave.
\end{cor}
\begin{proof}
This follows from Corollary~\ref{cor:lip} and the pointwise convergence of the $\f_n$ to  $\f$.
\end{proof}

Recall that in the course of proving Lemma~\ref{lem:2.8}, the following was proved.
\begin{claim}\lab{cl:2.4}
Let $C <  n_1^2 < n_2$. Then, \beqs \left(1 + \frac{C\log n_1}{n_1}\right) |P_{n_2}(s)|^{\frac{1}{n_2^2 - 1}} \geq |P_{n_1}(s)|^{\frac{1}{n_1^2 - 1}}. \eeqs 
\end{claim}
In light of Lemma~\ref{lem:2.8}, this has the following corollary.
\begin{cor}\lab{cor:2.5}
 \beqs \mathbf{f}_n(s) \leq \left(1 + \frac{C\log n}{n}\right) \mathbf{f}(s) . \eeqs 
\end{cor}

We will need the following claim in addition to Claim~\ref{cl:2.4}.

\begin{claim}\lab{cl:3.1}
Let $ n_1 \leq C' (\sqrt{n})$. Then, \beqs \left(1 - \frac{C\log n_1}{n_1}\right) |P_{n}(s)|^{\frac{1}{n^2 - 1}} \leq |P_{n_1}(s)|^{\frac{1}{n_1^2 - 1}}. \eeqs 
\end{claim}
\begin{proof}
Let $o = 0$, and $\square_{ij}^o$ be given by (\ref{eq:sqij}), where $n_2$ is the largest multiple of $n_1$ that is less or equal to $n$.
Since the push forward of a log concave density via a surjective linear map is a log-concave density by \cite{prekopa}, we see that the push forward of the uniform measure on $P_n(s)$ onto $\R^\b$ via the natural projection $\pi$ of  $\R^{V(\T_n)}$ onto $\R^\b$ is a log-concave measure. Taking into account that the subspace of mean zero functions maps surjectively onto $\R^\b$, we see that this measure is in fact absolutely continuous with respect to the Lebesgue measure and is thus a density, which we denote by $\rho$. Let $\rho$ be convolved with the indicator of an origin symmetric cube $Q$ of sidelength $\frac{1}{2M} = n^{-6}$, and let the resulting density be denoted by $\rho'$. Since the convolution of two log-concave densities of log-concave, we see that $\rho'$ is a log-concave density. However, $\rho'(x)$ is the measure that $\rho$ assigns to $Q+x$.
By Fradelizi's theorem, the value of log-concave density $\rho'$ on $\R^{\b}$ at its mean $0$ is no less than $e^{-|\b|}$ times the density at a mode. Thus, for every $x \in \R^{\b}$, we have $\rho'(x) \leq e^{|\b|} \rho'(0).$ Let $z$ be a point sampled from $Q$ from the measure obtained by restricting $\rho$ to $Q$ and normalizing it to a probability distribution $\rho_Q$. Consider the polytope $\pi^{-1}(z) \cap P_n(s)$ equipped with the conditional density, which is simply the uniform measure on $\pi^{-1}(z) \cap P_n(s)$ . Let us sample a point $z'$ from the uniform measure on $\pi^{-1}(z) \cap P_n(s)$. We claim that with probability at least $\frac{1}{2}$, for each $1 \leq i, j \leq n_2^2/n_1^2$,  $z'|_{\square_{ij}^o}$ corresponds to a point in $P_{n_1}(s)$, via the natural identification of $\square_{ij}^o$ with $\T_{n_1}$, after subtracting the mean.  At least $\frac{1}{2}$ the mass of $ker(\pi) \cap P_n(s)$ lies inside $(1 - M^{-1})\left(ker(\pi) \cap P_n(s)\right)$, and the distance of any point in  $(1 - M^{-1})\left(ker(\pi) \cap P_n(s)\right)$ to the boundary of $ker(\pi) \cap P_n(s)$ is at least $M^{-1}$. This follows from the convexity of $\pi^{-1}(z) \cap P_n(s)$ and the fact that $ker(\pi) \cap P_n(s)$ contains the unit ball in $ker(\pi)$ centered at the origin.

By our claim, $$M^{-|\b|} |P_{n_1}(s)|^{(n_2/n_1)^2} \geq \frac{e^{- |\b|}}{2} (Cn_1)^{- \frac{2n_2^2}{n_1^2}}|P_n(s)|.$$ This yields 
$$n^{- C n_1/n}|P_{n_1}(s)|^{(1/n_1)^2} \geq \frac{e^{- |\b|/n_2^2}}{2} (Cn_1)^{- \frac{2}{n_1^2}}|P_n(s)|^{(1/n_2^2)},$$ and since $|\b| = \Theta(n_2 n/n_1 + nn_1) = \Theta(n_1^3)$, the lemma follows.
\end{proof}
Consequently, taking limits on the left, and incorporating Corollary~\ref{cor:2.5} we have the following corollary.
\begin{cor}\lab{cor:...}
\beqs \left(1 - \frac{C\log n}{n}\right)  \mathbf{f}(s) \leq \mathbf{f}_n(s) \leq  \left(1 + \frac{C\log n}{n}\right)  \mathbf{f}(s). \eeqs
\end{cor}
\subsection{Surface area of facets of $P_n(s)$}
\begin{lem} \lab{lem:ess} There is a universal constant $C > 1$ such that 
for all sufficiently large $n$, the surface area of a codimension $1$ facet of $P_n(s)$ corresponding to $E_i(\T_n)$ is bounded below by $\left(\frac{s_0}{Cs_2}\right)^{\frac{Cs_i}{s_0}} |P_n(s)|^{1 - \frac{1}{n^2 - 1}}.$
\end{lem}

\begin{proof}
 Let $s$ be rescaled by scalar multiplication so that $|P_n(s)| = 1$. Knowing that $|P_n(s)|^{\frac{1}{n^2-1}}$ exists and and has a limit and lies in $[s_0, 2e s_0]$ , we see that $|P_n(s)|^{1 - \frac{1}{n^2 - 1}} \in [\frac{1}{2es_0}, \frac{1}{s_0}].$ Let $F_i$ denote a codimension $1$ facet corresponding to an edge in $E_i(\T_n)$. 
For all sufficiently small $\eps > 0$, we will find a lower bound on the probability that there exists a point $y \in F_i$ such that $\|y - x\|_{\ell_2} < \eps$, when $x$ is sampled at random from $P_n(s)$.
We identify $V(\T_n)$ with  $\Z/n\Z \times \Z/n\Z$ via the unique $\Z$ module isomorphism that maps $[\omega^i]$ to $(1, 0)$ and $[\omega^i \exp(\frac{\pi\imath}{3})]$ to $(0, 1)$. This causes the edges obtained by translating  $\{(0,0), (1, 0), (1, 1), (0,1)\}$ to belong to $E_i(\T_n)$. We further identify $\Z/n\Z \times \Z/n\Z$  with the subset of $\Z^2$ having coordinates in $(-\frac{n}{2}, \frac{n}{2}]$.  Let $T$ be the set of vertices contained in the line segment $\{(a, b)|(a = b) \,\mathrm{and}\,  (|a| \leq \frac{3s_i}{s_0})\}.$ Let $S$ be the set of all  lattice points (vertices) within the convex set $\{(a, b)|(|a -b| \leq 3) \,\mathrm{and} \, (|a + b| \leq \frac{6s_i}{s_0} + 3)\}$ that do not belong to $T$. Without loss of generality, we assume that $F_i$ corresponds to the constraint  $- x(0,0) + x(1, 0) -  x(1, 1) +  x(0,1) \leq s_i$. Let $conv(X)$ be used to denote the convex hull of $X$ for a set of bounded diameter. Let $U = \{u_{-2}, u_{-1}, u_0\}$ be a set of three adjacent vertices not contained in $S\cup T$, but such that exactly two of these vertices are respectively adjacent to  two distinct vertices in $S$. That such a $U$ exists follows from the presence of long line segments in the boundary of $conv(S\cup T)$. Given $x \in P_n(s)$, we define $x_{lin}:conv(U\cup S\cup T) \ra \R$ to be the unique  affine map from the convex hull of $U\cup S\cup T$ to $\R$ which agrees with the values of $x$ on $U$. The function $x_{lin}$ will serve as a baseline for the measurement of fluctuations.  Let $\Lambda_T$ denote the event that 
$\forall (a, a) \in T,$ 

$$ \left|x((a, a)) - x_{lin}((a, a)) - \min\left( \frac{  \left(|a-\frac{1}{2}| - \frac{1}{2}\right)s_0 -  2 s_i}{2}, 0\right)\right|  \leq  \frac{s_0}{20}.$$ 
Let $\Lambda_S$ be the event that for each vertex $v \in S,$ we have   $$-\frac{s_0}{100} \leq x(v)  - x_{lin}(v)  \leq  \frac{s_0}{100}.$$  Let $x_S$ denote the restriction of $x$ to $S$, and likewise define $x_T$, $x_{S\cup T}$ etc.  Let the  cube in $\R^S$ corresponding  to the event $\Lambda_S$ be denoted $Q_S$. Let the polytope in  $\R^T$ corresponding to the event $\Lambda_T$ be denoted $Q_T$. Note that $Q_T$ implicitly depends on $x_S$, but only through the effect of the one constraint $F_i$.
Let $z_S$ be a point in   $[-\frac{s_0}{100}, \frac{s_0}{100}]^S.$ Due to a double layer of separation between $T$ and $V(\T_n)\setminus S$, conditioned on $x_S$ being equal to  $z_S$, the distribution of $x_T$ is independent of the distribution of $x_{V(\T_n)\setminus S}.$ Also, conditioned on $x_s = z_s$, the distribution of $x_T$ is the uniform distribution on a $|T|$ dimensional truncated cube, of sidelength $\frac{s_0}{10}$, the truncation being due the linear constraint $$\langle x_T, \zeta_S\rangle  \geq x((1, 0)) + x((0, 1))- s_i$$ imposed by $F_i$, where $\zeta_S$ is a  vector in $\R^T$ (taking values $1$ on $\{(0,0), (1, 1)\}$ each and $0$ elsewhere). The euclidean distance of the center of this cube to $F_i$ is less than $\frac{s_0}{50}$, so together with Vaalar's theorem \cite{Vaaler} bounding the volume of a central section of a unit cube  from below by $1$, we see that  conditioned on  $\Lambda_T$ and $\Lambda_S$,  the probability that the distance of $x$ to $F_i$ is less than $\eps$ is at least 
$\eps 2^{-|T|}$ for all sufficiently small $\eps$. It remains for us to obtain a positive lower bound on $\p[\Lambda_S \, \mathrm{and} \, \Lambda_T]$ that is independent of $n$ for sufficiently large $n$.
Note that \beq \p[\Lambda_S \, \mathrm{and} \, \Lambda_T] =  \p[\Lambda_T| \Lambda_S]\p[\Lambda_S].\eeq
Let $\mu_{\Lambda_S}$ denote the conditional probability distribution of $x_S$ (supported on $Q_S$) given $\Lambda_S$.
\beqs \p[\Lambda_T| \Lambda_S] & = & \int \p[x_T \in Q_T| x_S = z_S ]\mu_{\Lambda_S}(dz_S)\\
& \geq & \inf_{z_S \in Q_S}\p[x_T \in Q_T| x_S = z_S ].\eeqs
Let $z_S \in Q_S$. Then, the conditional distribution of $x_T$ given that $x_S = z_S$ is the uniform (with respect to Lebesgue) measure on a polytope that is contained in the set of all vectors in  $\R^T$ which when augmented with $z_S$ are $2s_0$ Lipschitz when viewed as functions on ${S\cup T}.$ The latter polytope has volume at most $(4s_0)^{|T|}.$ Since $Q_T$, for any $z_S$, contains a unit cube of side length $s_0/100$, 
\beq  \p[\Lambda_T| \Lambda_S] & \geq & \inf_{z_S \in Q_S}\p[x_T \in Q_T| x_S = z_S ] \geq 400^{-|T|}.\eeq
Finally, we obtain a lower bound on $\p[\Lambda_S].$ 
We say that a vertex $v \in S$ is reachable from $U$ if there is a sequence of  vertices $u_{-2}, u_{-1}, u_0, v_1, \dots, v_k = v$ such that any $4$ consecutive vertices form an edge in $E(\T_n)$ and $v_0, \dots, v_k \in S$. By our construction of $U$, every vertex in $S$ is reachable from $U$, and the length of the path is at most $2|T| + 10$. Consider the values  of $x - x_{lin}$ on $S$. These values cannot exceed $(2|T| + 10)s_2$. Their mean is $0$. Their joint distribution has a density $g_S$ that is logconcave by Pr\'{e}kopa's Theorem~\ref{thm:prekopa}. The probability that $(x - x_{lin})_S$ lies in a translate of $Q_S$ by $t$ is equal to the value of the convolution of $g_S$ with the indicator $I(Q_S)$ of $Q_S$ at $t$. Multiplying by $\left(\frac{50}{s_0}\right)^{|S|}$ (to have unit $L_1$ norm), it follows that each coordinate in any point of the support of $\left(\frac{50}{s_0}\right)^{|S|} I(Q_S) \ast g$ is bounded above by $(2|T| + 12)s_i$, while the mean of this distribution continues to be $0$. The (differential) entropy of $g$ is bounded above by the sum of the entropies of its one dimensional marginals along coordinate directions, which in turn is bounded above by $\ln\left(2e(2|T| + 11)s_2\right)$ by Lemma~\ref{lem:5}. It follows that the supremum of the density of $\left(\frac{50}{s_0}\right)^{|S|} I(Q_S) \ast g$ is at least $\left(2e(2|T| + 12)s_2\right)^{-|S|}$. It is a theorem of Fradelizi \cite{Fradelizi} that the density at the center of mass of a logconcave density on $\R^{|S|}$ is no less than $e^{- |S|}$ multiplied by the supremum of the density. Applied to  $ I(Q_S) \ast g$, this implies that  $$\p[\Lambda_S] \geq \left(100e^2(2|T| + 11)\left(\frac{s_2}{s_0}\right)\right)^{-|S|}.$$
This shows that there is a universal constant $C > 1$ such that 
for all sufficiently large $n$, the surface area of a codimension $1$ facet of $P_n(s)$ corresponding to $E_i(\T_n)$ is bounded below by $\left(\frac{s_0}{Cs_2}\right)^{\frac{Cs_i}{s_0}} |P_n(s)|^{1 - \frac{1}{n^2 - 1}}.$
\end{proof}
\begin{lem} \lab{lem:surf_upperbd} Fix $s$ with $0 < s_0 \leq s_1 \leq s_2$ and $\eps > 0$, for all sufficiently large $n$, the surface area of a codimension $1$ facet of $P_n(s)$ corresponding to $E_i(\T_n)$ is bounded above  by $\left(\frac{(2e +\eps) s_0}{s_i}\right)|P_n(s)|^{1 - \frac{1}{n^2 - 1}}.$
\end{lem}
\begin{proof} Note that $$\sum_i \left(1 - \frac{1}{n^2}\right)^{-1} s_i \wn_i = |P_n(s)|,$$ which in turn is bounded above by $(2e +\eps) s_0|P_n(s)|^{1 - \frac{1}{n^2 - 1}}$ for sufficiently large $n$.
It follows for each $i \in \{0, 1, 2\}$, that $w_i^{(n)}$ is bounded above by  $\left(\frac{(2e +\eps) s_0}{s_i}\right)|P_n(s)|^{1 - \frac{1}{n^2 - 1}}.$ This completes the proof of this lemma.
\end{proof}

\subsection{A lower bound on the $\ell_\infty$ diameter of $P_n(s)$}
\begin{lem}\lab{lem:diameter}
The $\ell_\infty$ diameter of $P_n(s)$ is greater than $(s_1 + s_2)\lfloor n/2\rfloor^2/4$ for all $n$ greater than $1$.
\end{lem}
\begin{proof}
Recall from Lemma~\ref{lem:2.3} that there is a unique quadratic function $q$ from $\mathbb L$ to $\R$ such that $\nabla^2q$ satisfies the following.

\ben 
\item  $\nabla^2q(e) = - s_0,$ if $e \in E_0(\mathbb L)$.
\item $\nabla^2q(e)  =  - s_1,$  if $e \in E_1(\mathbb L )$.
\item $\nabla^2q(e)  =  - s_2,$ if  $e \in E_2(\mathbb L)$.
\item $q(0) = q(n) = q(n\omega) = 0$.
\een
We define the function $r$ from $\R^2$ to $\R$ to be the unique function that agrees with $q$ on $n \mathbb L$, but is defined at all points of $\R^2\setminus n{\mathbb L}$ by piecewise linear extension.
In other words, the epigraph of $-r$ is the convex hull of all points of the form $(v, -q(v))$ as $v$ ranges over $n \mathbb L$. 
The function $r - q$ restricted to $\mathbb L$ is invariant under shifts by elements in $n \mathbb L$ and so  can be viewed as a function from $V(\T_n)$ to $\R$. The function from $V(\T_n)$ to $\R$ obtained by adding a suitable constant $\kappa$ to $r - q$ such that it has zero mean is a member of $P_n(s)$. We readily see, by examining one of the sides of a fundamental triangle in $n\mathbb L$ that $\|r- q + \kappa\|_{\ell_\infty}$ is at least $(s_1 + s_2)\lfloor n/2\rfloor^2/4$. Since the constant function taking value $0$ belongs to $P_n(s)$, the lemma follows.
\end{proof}

\section{An upper bound on $|P_n(s)|$}

By known results on vector partition functions \cite{Brion}, $P_n(s)$ is a piecewise polynomial function of $s$, and each domain of polynomiality is a closed  cone known as a chamber of the associated vector partition function. For a different perspective, see also Lemma $2$ of \cite{Klartag}. It follows by scaling, that these polynomials are homogenous, of degree $n^2 - 1$. Further in the cone $\min(s_0, s_1, s_2) > 0$, $|P_n(s)|$ is differentiable, since the facets of $P_n(s)$ have finite volume.
 

Recall that \beq\lab{eq:wn}\frac{|P_n(s)|^{-1}}{n^2} \left(\frac{\partial|P_n(s)|}{\partial s_0}, \frac{\partial|P_n(s)|}{\partial s_1}, \frac{\partial|P_n(s)|}{\partial s_2}\right) =: |P_n(s)|^{-1} (w_0^{(n)}, w_1^{(n)}, w_2^{(n)}).\eeq

Let $\Delta_w$ be the function from $V(\T_{n})$ to $\R$, uniquely specified by the following condition. For any $f:V(\T_{n}) \ra \R$, and $(v_1, v_2) = v \in V(\T_{n})$,
\beq 2|P_n(s)|(\Delta_w \ast f)(v) & = & w_0^{(n)}(D_0f(v_1 -1, v_2 -1) + D_0f(v_1 -1, v_2))\nonumber\\
& + & w_1^{(n)}(D_1f(v_1, v_2 ) + D_1f(v_1 -1, v_2-1))\nonumber\\
& + & w_2^{(n)}(D_2f(v_1 -1, v_2 -1) + D_2f(v_1, v_2-1)).\lab{eq:7.7}\eeq
Note that $\Delta_w$ can be viewed as a self adjoint operator acting on $\mathbb{C}^{V(\T_{n})}$ equipped with the standard inner product. 


Given a self adjoint  linear operator $A$ from $\R^{V(\T_n)}$ to itself, that maps the linear subspace of mean zero functions (which we denote by $\R^m$) to itself,
we define $|A|$ to be the absolute value of the product of the eigenvalues of $A$ restricted to $\R^m$. 

\begin{lem}\lab{lem:4.1}  
 \beq |\De_w|^{ \frac{ 1}{m}} |P_n(s)|^{\frac{1}{m}} \leq e  + o_n(1).\eeq
\end{lem}
\begin{proof}
Let $L$ denote $P_n(s)$ and $K(g)$ denote $L \cap (g + L)$, where $g$ belongs to the span of $P_n(s)$ which we identify with $\R^m$.
By convolving the indicator of $L$ with that of $-L$,  we see that 
\beq \int\limits_{x \in L + (-L)} |L \cap (x + L)| dx = |L|^2. \lab{eq:5.1}\eeq

Recall that $$S_K(L) := \lim\limits_{\eps \ra 0} \frac{|L + \eps K| - |L|}{\eps}.$$
Recall from (\ref{eq:2.2}) that 
\beqs \frac{|K|}{|L|} \leq \left(\frac{S_K(L)}{m|L|}\right)^m. \eeqs

We define $S_{K-L}(L) := S_K(L) - m|L|,$ which since $K \subseteq L$, is a nonpositive real number. Then,

 \beqs \frac{|K|}{|L|} & \leq &  \left(\frac{S_K(L)}{m|L|}\right)^m.\eeqs

Let us define the ``negative part" of the Hessian of $g$, denoted $(\nabla^2g)_\_$ to be the real valued function on the edges (unit rhombi) $e$ in $\T_n$ such that \beq (\nabla^2g)_\_ (e) = \min(0, \nabla^2g(e)).\eeq

Recall from (\ref{eq:wn}) that \beqs \frac{1}{n^2} \left(\frac{\partial|P_n(s)|}{\partial s_0}, \frac{\partial|P_n(s)|}{\partial s_1}, \frac{\partial|P_n(s)|}{\partial s_2}\right) =: (w_0^{(n)}, w_1^{(n)}, w_2^{(n)}).\eeqs 
Let $W$ denote the operator that maps $h:E(\T_n) \ra \R$ to $W h$, where for $e \in E_r(\T_n)$, we define $w(e)$ to be $ w_r^{(n)}$ and set  $W h(e) = w(e) h(e).$ For $e \in E_r(\T_n)$, we define $s(e)$ to be $s_r$. 

It follows that  \beqs S_{K-L}(L) & = &  \sum\limits_{e \in E(\T_n)} w(e)(\nabla^2g)_\_(e)\\
& = & (-1)\frac{\|W\nabla^2 g\|_1}{2}.\eeqs

We note that the map ${W\nabla^2g} \mapsto  \De_w g$ is a contraction in the respective $\ell_1$ norms. 
It follows from (\ref{eq:5.1}) that 
\beq |L|^2 & = & \int\limits_{L + (-L)} |K(g)| dg\nonumber\\
& \leq & \int\limits_{L + (-L)} |L| \left(\frac{S_{K}(L)}{m|L|}\right)^{m} dg\nonumber\\
& = &   \int\limits_{L + (-L)} |L| \left(1 -\frac{{\|W\nabla^2 g\|_1}}{2m|L|}\right)^m dg\nonumber\\
& \leq & \int\limits_{L + (-L)} |L| \left(1 -\frac{{\||L| \De_w g\|_1}}{2m|L|}\right)^m dg.\lab{eq:L1}\eeq
We thus see that \beq |L| & \leq & \int\limits_{L + (-L)} \left(1-\frac{{\| \De_w g\|_1}}{2m}\right)^m dg\\
& \leq &  \int\limits_{\R^m \cap  \{\|f\|_1 \leq 2m\}} \left(1 -\frac{{\|f\|_1}}{2m}\right)^m | \De_w|^{-1}df.\eeq
We see by a packing argument that
\beqs  \left(\int\limits_{\R^m \cap \{\|f\|_1 \leq 2m\}} \left(1 -\frac{{\|f\|_1}}{2m}\right)^m df\right)^\frac{1}{m} & \leq &   \left(\int\limits_{\{\|f\|_1 \leq 2m\}} \left(1 -\frac{{\|f\|_1}}{2m}\right)^m df\right)^\frac{1}{m},\eeqs
where the last integral is over an $\ell_1$ ball of radius $2m$ contained in $\R^{V(\T_n)}$.

We evaluate $$  \left(\int\limits_{\{\|f\|_1 \leq 2m\}} \left(1 -\frac{{\|f\|_1}}{2m}\right)^m df\right)^\frac{1}{m}$$ by integrating over the boundaries of $\ell_1$ balls  of increasing radius as follows.
Let $V_1(d)$ denote the volume of a unit $\ell_1$ ball in $\R^d$. We observe that$$\lim_{d \ra \infty} d\left( V_1(d)^\frac{1}{d}\right)  = \lim_{d \ra \infty} d \left(\frac{2^d}{d!}\right)^{\frac{1}{d}} =   2e.$$
\beqs \left(\int\limits_{\{\|f\|_1 \leq 2m\}} \left(1 -\frac{{\|f\|_1}}{2m}\right)^m df\right)^\frac{1}{m} & = & 
\left(\int_0^{2m} \int\limits_{\{\|f\|_1 = t\}} \left(1 -\frac{t}{2m}\right)^m df dt \right)^\frac{1}{m}\\
& = & 
\left(\int_0^{2m} \int\limits_{\{\|f\|_1 \leq t\}} \left(\frac{\sqrt{m}}{t}\right)\left(1 -\frac{t}{2m}\right)^m df dt \right)^\frac{1}{m}\\
& \leq & 
\left(\sup_{t \in [0, 2m]}\left(\frac{2m \sqrt{m}}{t}\right)\left(t -\frac{t^2}{2m}\right)^m V_1(m+1) dt \right)^\frac{1}{m}\\
& \leq & 
\left(\sup_{t \in [0, 2m]}\left(\frac{2m \sqrt{m}}{t}\right)\left(t - \frac{t^2}{2m}\right)\left(\frac{m}{2}\right)^{m-1} V_1(m+1) dt \right)^\frac{1}{m}.\eeqs
Therefore, we see that \beq \lab{eq:2.29} |L|^{\frac{1}{m}}| \De_w|^{\frac{1}{m}} \leq (e + o(1)).\eeq

This proves the lemma.
\end{proof}

Let $\C$ denote the open cone in $\R_+^3$ consisting of points $\tilde{u}=(\tilde{u}_0, \tilde{u}_1, \tilde{u}_2)$ such that  
$$\min_\sigma\left(\tilde{u}_{\sigma(0)} + \tilde{u}_{\sigma(1)}- \tilde{u}_{\sigma(2)}\right) > 0,$$ 
where $\sigma$ ranges over all permutations of $\{0, 1, 2\}$. 

Note that the expression $$(a + b - c)(a - b + c) + (a + b - c)(-a + b + c) + (a - b + c)(-a + b + c),$$ simplifies to
$$(-1)(a^2 + b^2 + c^2) +  2ab + 2bc + 2ca.$$
Thus, we see that every point $\tilde{u} \in \C$ also satisfies 
\beqs \tilde{u}_0^2 + \tilde{u}_1^2 + \tilde{u}_2^2 < 2\left(\tilde{u}_0 \tilde{u}_1 + \tilde{u}_1 \tilde{u}_2 + \tilde{u}_2 \tilde{u}_0\right).\eeqs

When $(w_0^{(n)}, w_1^{(n)}, w_2^{(n)}) \in \C$, a theorem of Kenyon (see Theorem 1.1 in \cite{Kenyonlogdet}) shows us how to estimate $ |\De_w|^{\frac{1}{m}}$ asymptotically.

Note that in this limit we keep $$\frac{(\wn_0, \wn_1, \wn_2)}{2|P_n(s)|}$$ constant as $n \ra \infty$. In the process $s$ may vary as a function of $n$. Such  $s$ exist for all sufficiently large $n$, by Minkowski's theorem \cite{Klain} for polytopes, which is stated below.
\begin{thm}\lab{thm:Mink}{\bf (Minkowski)}
Suppose $\e_1, \e_2, \dots, \e_k$ are unit
vectors that do not all lie in a hyperplane of positive codimension, and suppose that $\a_1, \a_2, \dots, \a_k >0.$ If
$\sum_i \a_i \e_i = 0$
then there exists a polytope $P_n$ having facet unit normals $\e_1, \e_2, \dots, \e_k$ and
corresponding facet areas $\a_1, \dots, \a_k$. This polytope is unique up to translation.
\end{thm}
Suppose  that $$(\wn_0)^2 + (\wn_1)^2 + (\wn_2)^2 - 2\left(\wn_0\wn_1  + \wn_1\wn_2 + \wn_2\wn_0\right) = (-4) |P_n(s)|^2.$$ This can be achieved by multiplying $\wn$ by a suitable positive scalar, since $\wn \in \C$. We define $\tw_i$ to be $\frac{\wn_i}{2|P_n(s)|}.$
Setting $$ \tan \theta_i := \tw_0+ \tw_1 + \tw_2 - 2 \tw_i,$$ for $i = 0, 1, 2$, where $\theta_i \in [0, \pi/2]$, we see that $$\tan \theta_0 \tan \theta_1 + \tan \theta_1 \tan\theta_2 + \tan \theta_2 \tan \theta_0 = 1$$  because
 \beq \tw_0^2 + \tw_1^2 + \tw_2^2 - 2\left(\tw_0\tw_1  + \tw_1\tw_2 + \tw_2\tw_0\right) = -1.\lab{eq:w123}\eeq This implies that  
 $$\tan(\theta_0 + \theta_1) = \frac{\tan \theta_0 + \tan \theta_1}{1 - \tan \theta_0 \tan \theta_1} = \cot \theta_2 = \tan\left(\frac{\pi}{2} - \theta_2\right).$$
 Therefore, $\theta_0 + \theta_1 + \theta_2 = \frac{\pi}{2}$. Since the $\theta_i \in [0, \frac{\pi}{2}]$, 
 giving the graph edges corresponding to diagonals of rhombic hyperedges in $E_i$,  weight $\tan \theta_i,$ gives rise to an isoradial embedding in the sense of \cite{Kenyonlogdet}. These graph edges  with weights $\tan \theta_0, \tan \theta_1$ and $\tan \theta_2$ correspond respectively to the sides $BC$, $AC$ and $AB$ in the above figure.
 With this notation, by Theorem 1.1 of \cite{Kenyonlogdet}, we have
 \beqs \lim_{n \ra \infty}  |\De_w|^\frac{1}{m} = \exp\left(\frac{2}{\pi}\sum_{i=0}^2 \left(-\int_0^{\theta_i} \log (2\sin t) dt - \int_0^{\frac{\pi}2 - \theta_i} \log(2 \sin t) dt + \theta_i \log \tan(\theta_i)\right)\right). \eeqs

\begin{figure}
\begin{center}
\includegraphics[scale=0.4]{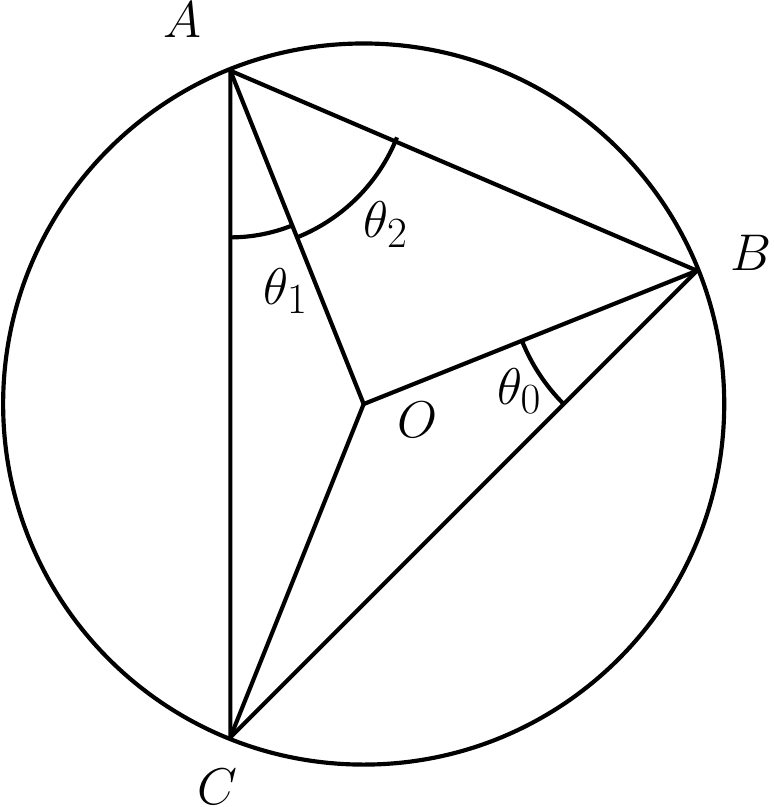}
\caption{$\theta_i$ in an isoradial embedding of one triangle of the equilateral lattice.}
\end{center}
\label{fig:circum}
\end{figure}

We thus have the following corollary to Lemma~\ref{lem:4.1}.
\begin{cor}
Suppose $\tw \in \C$ is fixed and satisfies the normalization condition (\ref{eq:w123}), and $s = s^{(n)}$ as a function of $n$ varies correspondingly. Then, 
 \beqs  \limsup_{n \ra \infty} |P_n(s^{(n)})|^{\frac{1}{m}} \leq \exp\left(1 + \frac{2}{\pi}\sum_{i=0}^2 \left(\int_0^{\theta_i} \log (2\sin t) dt + \int_0^{\frac{\pi}2 - \theta_i} \log(2 \sin t) dt - \theta_i \log \tan(\theta_i)\right)\right).\eeqs
\end{cor}

\section{Various norms}
\subsection{Bounds on the $\ell_p$ norm of a point in $P_n(s)$}

Our strategy will be to cover the set of points in $P_n(s)$ that are far from the origin  by a small number of polytopes, each of which is of small volume.

\begin{lem}\lab{lem:infty}
Suppose that $\eps_0 > 0$ and $2 = s_0 \leq s_1 \leq s_2.$ Let $x \in P_n(s)$ be such that 
$\|x\|_{\infty} \geq \eps_0 n^2.$ Then, for any $p \in [1, \infty)$,
\beq\|x\|_{p} \geq \left(\frac{\sqrt{3}\eps_0 n}{8 s_2}\right)^{\frac{2}p}\left(\frac{\eps_0 n^2}{2}\right).\eeq
\end{lem}
\begin{proof}
Let the magnitude of the slope of $x$ on a unit triangle $t $ with vertices $v_i, v_j, v_k$ in $\T_n$ be defined to be $\max(|x(v_i) - x(v_j)|, |x(v_j) - x(v_k)|, |x(v_k) - x(v_i)|)$.
 Choose $v_- \in \T_n$ such that $x(v_-)$ is minimal and $v_+ \in \T_n$ such that $x(v_+)$ is maximal.
 Note that the magnitude  of the slope of a triangle $t$ containing $v_-$ cannot exceed $s_2$ because the discrete Hessian of all the rhombi containing $v_-$ are bounded above by $s_2$. It is possible to go from one unit triangle with vertices in $\T_n$ to $v_-$ via a sequence of vertices, every $4$ consecutive vertices of which form a unit rhombus,  such that the total number of rhombi is less than $4n$. For this reason the slope of $x$ at no unit triangle can exceed $4ns_2$ in magnitude. Let $v = v_+$ if $x(v_+) \geq - x(v_-)$ and $v = v_-$ otherwise. Therefore, 
$\|x\|_{\infty} \geq \eps_0 n^2$ implies that any vertex $\widehat v$  within a lattice distance of 
$\frac{\eps_0 n^2}{8ns_2}$ of $v$ satisfies $\frac{x(\widehat v)}{x(v)} > \frac{1}{2},$ implying that $|x(\widehat{v})| \geq \frac{\eps_0 n^2}{2}.$  The number of vertices within a lattice distance of $\frac{\eps_0 n^2}{8ns_2}$ of $v$ is at least $3\left(\frac{\eps_0 n}{8 s_2}\right)^2$. Therefore, \beq\|x\|_{p}^p \geq 3\left(\frac{\eps_0 n}{8 s_2}\right)^2\left(\frac{\eps_0 n^2}{2}\right)^p.\eeq This implies the lemma.
\end{proof}

\subsection{Discrete Sobolev norms}

\begin{defn} For $g \in \R^{V(\T_n)}$, we define the discrete Sobolev (semi-)norm $\|g\|_{L_p^2}$ by
\beq \|g\|_{L_p^2} := \left(\sum_{v \in V(\T_n)} \left(|D_0 g(v)|^p + |D_1 g(v)|^p + |D_2 g(v)|^p\right)\right)^{\frac{1}{p}}. \eeq
\end{defn}


A random variable $Z$ in $\R$ that satisfies for some positive real $K$, $$\E[\exp(|X|/K)] \leq 2$$ is called subexponential. 
\begin{defn}[$\psi_1$ norm]
We define $$\|Z\|_{\psi_1} = \inf\{ t > 0: \E[\exp(|X|/t)] \leq 2\}.$$
\end{defn}

Fix $i \in \{0, 1, 2\}$.  Let $v \in V(\T_n)$ and let the density of the distribution of $\frac{(-1) D_i g(v)}{s_i}$ be denoted by $f$.   We then see that $f$ is independent of the specific $v$ chosen, (by the transitive action of $\T_n$ on itself) and have \beq \int_\R x f(x) dx = 0.\eeq and \beq\lab{2.10}\int_\R f(x) dx = 1.\eeq

Since $f$ is bounded from above and is continuous, it achieves its supremum. Let $x_0 \in \R$ satisfy \beq f(x_0) = \sup_{x \in \R} f(x) =: c_0. \eeq
\begin{lem}\lab{lem:5.4-22nov}
For all $x \geq x_0 + 6$,

\beq f(x) \leq \frac{2^{-\frac{x -( x_0 + 6)}{6}}}{6} \leq C \exp\left(- \left(\frac{\ln 2}{6}\right) x\right).\eeq

\end{lem}
\begin{proof}
We see that 
\beq\lab{2.12} \int_{-1}^0 f(x)dx \leq  c_0, \eeq therefore,
\beq \int_{-1}^0 xf(x)dx \geq  - c_0. \eeq
Since $f$ has mean $0$,  \beq \int_0^\infty x f(x) dx \leq c_0.\eeq
This implies that \beq\lab{2.15} \int_1^\infty f(x)dx \leq c_0.\eeq
It follows from (\ref{2.10}), (\ref{2.12}) and (\ref{2.15}) that 
\beq \int_0^1 f(x)dx \ge 1 - 2c_0. \eeq
Therefore, $c_0 \geq 1 - 2 c_0$, and so \beq c_0 = \sup_{x \in \R} f(x) = f(x_0) \leq \frac{1}3.\eeq
Suppose that $x_0 > 0$. 
Then, by the log-concavity of $f$,
\beq \int_0^{x_0} f(x) dx & \geq & f(0)\int_0^{x_0} \exp\left((x/x_0) \ln \frac{f(x_0)}{f(0)}\right)dx\\
                                           & = & f(0)\left( \frac{\exp\left((x/x_0) \ln \frac{f(x_0)}{f(0)}\right) x_0}{ \ln \frac{f(x_0)}{f(0)}}\Big|_0^{x_0}\right)\\
                                           & = & \left(\frac{x_0 f(0)}{ \ln \frac{f(x_0)}{f(0)}}\right)\left(\frac{f(x_0)}{f(0)} - 1\right).\eeq
                                           This implies that $x_0 \leq \frac{\ln f(x_0) - \ln f(0)}{f(x_0) - f(0)}.$ 
                                                          
As $f(0) \leq f(x_0)$, we see that by log-concavity of $f$, 
\beq \sup_{x \in[-1,0] } f(x) = f(0) = \inf_{x \in [0, x_0]} f(x). \eeq
 Since $f$ has zero mean, this implies that \beq x_0 \leq 1.\eeq
                                           
Since $f$ is monotonically decreasing on $[x_0, x_0 + 6]$, $f$ attains its minimum on this interval at $x_0 + 6$.
Therefore $f(x_0 + 6) \leq \frac{1}{6}.$ 
Again, by log-concavity, for all $x \geq x_0 + 6$,

\beq f(x) \leq \frac{2^{-\frac{x -( x_0 + 6)}{6}}}{6} \leq C \exp\left(- \left(\frac{\ln 2}{6}\right) x\right).\eeq 
\end{proof}

\begin{lem}\lab{lem:lemma1}
If $g$ is chosen uniformly at random from $P_n(s)$,   \beq \E \|g\|^p_{L_p^2} \leq K^p (ps_2)^p n^2,\eeq where $K$ is a universal constant.
\end{lem}
\begin{proof}
By Lemma~\ref{lem:5.4-22nov}, we see that $\frac{(-1) D_i g(v)}{s_i}$ is subexponential and $$\left\|\frac{(-1) D_i g(v)}{s_i}\right\|_{\psi_1} < C_2,$$ for some universal constant $C_2$.
Recalling that $s_0 \leq s_1 \leq s_2$, the lemma follows by the linearity of expectation and the fact (see Proposition 2.7.1 of \cite{Vershynin}) that the $p^{th}$ moments of a subexponential random variable $X$ satisfy \beq \E|X|^p \leq (C_2p)^p.\eeq for a universal constant $C_2$.
\end{proof}

We use this to derive the following.

\begin{lem}
If $g$ is chosen uniformly at random from $P_n(s)$, there is a universal constant $C_2$ such that for any $\de \in (0, e^{-1}),$
\beq \p\left[\|g\|_{L_2^2}  \geq C_2 n \ln \de^{-1}\right] \leq \de.\eeq
\end{lem}
\begin{proof}
We see that for any $p \geq 2$, by the monotonically increasing nature of the $\ell_p$ norms as $p$ increases, for each $g \in P_n(s)$ \beq \left(\frac{\|g\|^2_{L_2^2}}{n^2}\right)^{\frac{1}{2}} \leq \left(\frac{\|g\|^p_{L_p^2}}{n^2}\right)^{\frac{1}{p}}.\eeq
This implies that 
\beq  \E\left(\frac{\|g\|^2_{L_2^2}}{n^2}\right)^{\frac{p}{2}} \leq \E \left(\frac{\|g\|^p_{L_p^2}}{n^2}\right) \leq (C_2 p)^p\eeq
An application of Markov's inequality gives us  \beq \p\left[\left(\frac{\|g\|^2_{L_2^2}}{C_2^2 p^2 n^2}\right)^{\frac{p}{2}}  \geq R^p\right] \leq R^{-p}.\eeq
Simplifying this, we have \beq \p\left[\|g\|_{L_2^2}  \geq C_2 p n R\right] \leq R^{-p}.\eeq
Setting $R$ to $e$, and absorbing it into $C_2$ and setting $p$ to $\ln \de^{-1}$, we now have
\beqs \p\left[\|g\|_{L_2^2}  \geq C_2 n \ln \de^{-1}\right] \leq \de.\eeqs
\end{proof}

\begin{defn}
Let $$\|g\|_W :=  \frac{1}{2}\left(\sum_{v \in V(\T_n)} \left(|A_0^2 g(v)|^2 + |A_2^2 g(v)|^2\right)\right)^{\frac{1}{2}}.$$
\end{defn}

\begin{lem}\lab{lem:2.23-sep-7-2020}
Let $g \in P_n(s)$. Suppose that $\|g\|_{L^2_2} \leq C_2n \ln \de^{-1}.$ 
Then, 
$\|g\|_W \leq 2 \|g\|_{L_2^2} \leq 2C_2n \ln \de^{-1}.$
\end{lem}
\begin{proof}
We see that 
 \beq \frac{\|g\|_W}{2} & = &  \frac{1}{2}\left(\sum_{v \in V(\T_n)} \left(|A_0^2 g(v)|^2 + |A_2^2 g(v)|^2\right)\right)^{\frac{1}{2}}\\
& = &  \left(\frac{1}{4}\sum_{v \in V(\T_n)} \left(|D_0 g(v) + D_1 g(v)|^2 + |D_1 g(v) + D_2 g(v)|^2\right)\right)^{\frac{1}{2}}\\
& \leq &  \left(\sum_{v \in V(\T_n)} \left(|D_0 g(v)|^2 + |D_1 g(v)|^2 + |D_2 g(v)|^2\right)\right)^{\frac{1}{2}} \\  & = & \|g\|_{L_2^2}.\eeq
Therefore, we have 
$\|g\|_W \leq 2 \|g\|_{L_2^2} \leq 2C_2n \ln \de^{-1}.$
\end{proof}
\begin{lem}\lab{lem:5.8-Nov-2020}
Let $g$ be sampled from the uniform measure on $P_n(s)$. Then,  $$\p\left[\|g\|_\infty >  \left(\frac{\a \log n}{n}\right) \sqrt{\E \|g\|_2^2}\right] < n^{-c \a}.$$
\end{lem}
\begin{proof} 
The density  $\rho$ of $g(v)$ for a fixed vertex $v$ is logconcave and mean $0$. This density  is identical for each vertex $v$ by symmetry. It follows from the Chebychev inequality that the $\psi_1$ norm of the corresponding  random variable is at most $C \frac{\sqrt{\E\|g\|_2^2}}{n}.$ The lemma follows the from the exponential tail decay and the zero mean property of $\rho$.
\end{proof}
\subsection{Studying the fluctuations using characters}

\begin{defn}For $(\i, \j) \in (\Z/n\Z)\times(\Z/n\Z)$, and $\omega_n = \exp(2\pi\sqrt{-1}/n),$ let $\psi_{\i \j}$ be the character of  $(\Z/n\Z)\times(\Z/n\Z)$ given by $\psi_{\i \j}(i, j) := \omega_n^{\i i + \j j}.$
\end{defn}

These span the eigenspaces of any translation invariant linear operator on $\mathbb{C}^{V(\T_n)}.$
Let $g$ be expressed as a linear combination of the characters over $\mathbb{C}$ as 
\beq g = \sum_{\i, \j} \theta_{\i \j}\psi_{\i \j},\eeq where, since $g \in \R^{V(\T_n)}$, we have $\theta_{\i \j} = \bar{\theta}_{-\i\,\, -\j}.$

\begin{lem}\lab{lem:2.11} Let $g \in P_n(s)$. Suppose that $\|g\|_{L^2_2} \leq C_2n \ln \de^{-1},$ and that $\|g\|_{2} \geq \eps_0 n^3.$
This implies that 
there exists $(k_0, \ell_0) \in \Z^2$ such that $(k_0^2 + \ell_0^2) \leq \frac{C_2\log \de^{-1}}{\eps_0}$ and 
 \beq |\theta_{k_0 \ell_0}| & \geq & \frac{c \eps_0 n^2}{\sqrt{C_2\log \de^{-1}}}.\lab{eq:2.67}\eeq
\end{lem}
\begin{proof}
By the orthogonality of the characters,
\beq \|g\|_2^2 = \sum_{k, \ell} |\theta_{k\ell}|^2 \|\psi_{k\ell}\|_2^2 = n^2\sum_{k, \ell} |\theta_{k\ell}|^2 . \eeq

Also,
\beq\|g\|_{2} \geq  \eps_0 n^3.\eeq

Therefore, \beq \sum_{k, \ell} |\theta_{k\ell}|^2  \geq  \eps_0^2 n^4.\lab{eq:3.45}\eeq
By virtue of the fact that $A_0$ and $A_1$ commute with translations of the torus $\T_n$,
\beq (2\pi)^{-4}\|g\|_W^2 = \sum_{k, \ell} |\theta_{k\ell}|^2 \left(\frac{k^4 + \ell^4}{n^4}\right) \|\psi_{k\ell}\|_2^2 \geq\left(\frac{1}{2}\right)\sum_{k, \ell} |\theta_{k\ell}|^2 \left(\frac{k^2 + \ell^2}{n}\right)^2. \eeq

Therefore, by Lemma~\ref{lem:2.23-sep-7-2020}, we see that \beq \sum_{k, \ell} |\theta_{k\ell}|^2\left({k^2 + \ell^2}\right)^2  \leq C C_2^2n^4\left(\log \de^{-1}\right)^2,\eeq
We use (\ref{eq:3.45}) to get 
\beq \frac{\sum_{k, \ell} |\theta_{k\ell}|^2 \left({k^2 + \ell^2}\right)^2}{\sum_{k, \ell} |\theta_{k\ell}|^2} \leq \frac{C_2^2 \left(\log \de^{-1}\right)^2}{ \eps_0^2}.\eeq
Defining \beq \mu_{k\ell} :=  \frac{|\theta_{k\ell}|^2}{\sum_{k, \ell} |\theta_{k\ell}|^2},\eeq
and $X$ to be the random variable that takes the value $(k, \ell) \in \Z^2$ with probability $\mu_{k\ell}$,
we see that \beq \p\left[\|X\|_2^4 \leq  \frac{C_2^2\left(\log \de^{-1}\right)^2}{ \eps_0^2}\right]\geq \p\left[\|X\|_2^4 \leq 4 \E \|X\|_2^4\right] \geq  \frac{3}{4}.\eeq

It follows that there exists $(k_0, \ell_0)$ such that \beq (k_0^2 + \ell_0^2) \leq \frac{C_2\log \de^{-1}}{\eps_0}\lab{eq:change-22-11-2020}\eeq and  \beq \mu_{k_0 \ell_0} \geq \frac{c\eps_0}{C_2\log \de^{-1}},\eeq (since the probability is mostly distributed among the few $(k_0, \ell_0)$ that satisfy (\ref{eq:change-22-11-2020}).  This implies that 
there exists $(k_0, \ell_0) \in \Z^2$ such that $(k_0^2 + \ell_0^2) \leq \frac{C_2\log \de^{-1}}{\eps_0}$ and 
 \beq |\theta_{k_0 \ell_0}| & = & \sqrt{\mu_{k_0\ell_0} \sum_{k, \ell} |\theta_{k\ell}|^2}\\ & \geq &  \sqrt{\left(\frac{c\eps_0}{C_2\log \de^{-1}}\right) (\eps_0 n^4)}\\ & = & \frac{c \eps_0 n^2}{\sqrt{C_2\log \de^{-1}}}.\eeq
\end{proof}

In this section, we use $m$ to denote $n^2-1$, the dimension of $P_n(s)$.
\begin{lem}\lab{lem:2.14}
Let $g \in P_n(s)$. 
Let $g$ be expressed as a linear combination of the characters over $\mathbb{C}$ as 
\beqs g = \sum_{\i, \j} \theta_{\i \j}\psi_{\i \j},\eeqs where, since $g \in \R^{V(\T_n)}$, we have $\theta_{\i \j} = \bar{\theta}_{-\i\,\, -\j}.$ Let $K$ be any convex set in the space of functions $\R^{V(\T_n)}$ that is invariant under translations of the domain, that is, the torus $\T_n$. Then, for any $(k_0, \ell_0) \in (\Z/n\Z)^2$, 
\beq \left|P_n(s) \cap K \cap B_\infty(g, \eps_{0.5} n^2)\right| \leq   \left|P_n(s) \cap K \cap B_\infty\left( \Re\left(\theta_{k_0\ell_0}\psi_{k_0\ell_0}\right), \eps_{0.5} n^2\right)\right|.\eeq 
In particular, we may choose $(k_0, \ell_0)$, from the conclusion of Lemma~\ref{lem:2.11}.
\end{lem}
\begin{proof}
Note  by the orthogonality of characters of $\Z_n \times \Z_n$, that \beq g \ast \left(\frac{\psi_{k_0\ell_0} + \psi_{-k_0\, -\ell_0} + 2}{2 n^2}\right) & = & \left(\frac{1}{2}\right)\left(\theta_{k_0\ell_0}\psi_{k_0\ell_0} +\theta_{- k_0\, -\ell_0}\psi_{-k_0\, - \ell_0}\right)\nonumber\\
& = & \Re  \left(\theta_{k_0\ell_0}\psi_{k_0\ell_0}\right).\eeq

Also note that $\rho:= \left(\frac{\psi_{k_0\ell_0} + \psi_{-k_0\, -\ell_0} + 2}{2 n^2}\right)$ is a probability distribution supported on $V(\T_n)$.
For $x \in P_n(s)$, let $B_\infty(x, \eps_{0.5} n^2)$ denote the  $\ell_\infty$ ball with center $x$ and radius  $\eps_{0.5} n^2.$ Below, $\bigplus$ represents Minkowski sum.
For  $r$ points $x_1, \dots, x_r$ in $P_n(s)$ and any non-negative reals $\a_1, \dots, \a_r$ such that $\sum_i \a_i = 1$, let $x:= \sum_i \a_i x_i.$ By the Brunn-Minkowski inequality, the convexity of $P_n(s)\cap K $,

\beq\sum_i \a_i \left|P_n(s) \cap K \cap B_\infty(x_i, \eps_{0.5} n^2)\right|^\frac{1}{m} & \leq & \left|\bigplus_{i \in [r]}  \a_i \left(P_n(s) \cap K \cap B_\infty(x_i, \eps_{0.5} n^2)\right)\right|^{\frac{1}{m}}\nonumber\\
 & \leq &  \left|P_n(s) \cap K \cap B_\infty(x, \eps_{0.5} n^2)\right|^\frac{1}{m}.\nonumber\eeq

Suppose that $g \in P_n(s)$ and  $T_v g = g \ast\de_{v}$, where $\de_{v}:V(\T_n) \ra \R$ is the function that takes value $1$ on $v$ and value $0$ on all other points $v' \in V(\T_n)$.
Then, because $P_n(s) \cap K$ is left fixed by the action of the group $\Z_n \times \Z_n$ acting on $V(\T_n)$ by translation, we see that 
\beq  \left|P_n(s) \cap K \cap B_\infty(T_v g, \eps_{0.5} n^2)\right|  = \left|P_n(s) \cap K \cap B_\infty(g, \eps_{0.5} n^2)\right|.\eeq 
By the convexity of $P_n(s)\cap K $, 
\beq \bigplus_{v \in V(\T_n)}  \rho(v) \left(P_n(s) \cap K \cap B_\infty(T_v g, \eps_{0.5} n^2)\right) \subseteq P_n(s) \cap K \cap B_\infty(\rho \ast g, \eps_{0.5} n^2).\eeq

As a consequence, 
\beq\left|P_n(s) \cap K \cap B_\infty(g, \eps_{0.5} n^2)\right|^{\frac{1}{m}} & = & \sum_{v \in V(\T_n)}  \rho(v) \left|P_n(s) \cap K \cap B_\infty(T_v g, \eps_{0.5} n^2)\right|^\frac{1}{m}\nonumber \\ 
& \leq & \left|\bigplus_{v \in V(\T_n)}  \rho(v) \left(P_n(s) \cap K \cap B_\infty(T_v g, \eps_{0.5} n^2)\right)\right|^\frac{1}{m}\nonumber \\ 
& \leq &  \left|P_n(s) \cap K \cap B_\infty(\rho \ast g, \eps_{0.5} n^2)\right|^\frac{1}{m}\nonumber \\
& = &        \left|P_n(s) \cap K \cap  B_\infty\left( \Re\left(\theta_{k_0\ell_0}\psi_{k_0\ell_0}\right), \eps_{0.5} n^2\right)\right|^\frac{1}{m}         \nonumber                .\eeq

Therefore,
\beqs \left|P_n(s) \cap K \cap B_\infty(g, \eps_{0.5} n^2)\right| \leq   \left|P_n(s) \cap K \cap B_\infty\left( \Re\left(\theta_{k_0\ell_0}\psi_{k_0\ell_0}\right), \eps_{0.5} n^2\right)\right|.\eeqs
\end{proof}
\begin{defn}
For $f \in \R^{V(\T_n)}$, and $k \in \Z$, such that $k \geq 1$ let $$\|f\|_{\dot{\CC}^k} := \max_{r_1, \dots, r_k \in \{0, 2\}} \|A_{r_1} \dots A_{r_k} f\|_\infty.$$
\end{defn}
Suppose without loss of generality that $k_0 \geq \ell_0$. 
For the remainder of this paper, let $g =  \Re\left(\theta_{k_0\ell_0}\psi_{k_0\ell_0}\right).$

\begin{lem}\lab{lem:2.12} 
We have $$\|g\|_{\dot{\CC}^2} \leq C  s_2.$$
\end{lem}
\begin{proof}
We see that, because  $g$ belongs to $P_n(s)$,  \beqs \|g\|_{\dot{\CC}^2} & \leq &   C \min_{r\in \{0, 1, 2\}} \min_{v \in V(\T_n)} D_r g(v),\\
& \leq & C s_2.\eeqs 
\end{proof}

\begin{lem} \lab{lem:2.125}
Let $f$ be chosen uniformly at random from $P_n(s)$. Let $r \in \{0, 1, 2\}$ then, \beqs \p\left[\|D_r( f)\|_\infty > \check{C}\log n\right] < n^{-c\check{C} + 2},\eeqs for some universal constant $c > 0$.
\end{lem}
\begin{proof}
For any fixed $v$, \beqs \E\left[ D_r f(v)\right] = 0.\eeqs and $D_r f (v)$ has a logconcave density by Prekopa-Leindler inequality, which by the constraints of the polytope, has a support contained in $(-\infty, s_r]$.
The Lemma follows from the exponential tail bound satisfied by a logconcave density, as shown in the proof of  Lemma~\ref{lem:lemma1} together with an application of the union bound.
\end{proof}

\begin{lem}\lab{lem:2.13} 
For any $f \in B_\infty(g, \eps_{0.5}n^2)$, we have \beq \left(\frac{\|f - g\|_{\dot{\CC}^1}}{s_2}\right)^2 \leq C \eps_{0.5}n^2. \eeq
\end{lem}
\begin{proof}
By Lemma~\ref{lem:2.12}, we see that for $r \in \{0, 1, 2\},$ $D_r(f -g) \leq C s_2,$ and hence $\forall r, \, A_r^2 (f-g) \leq 2 C s_2$. Let $v \in V(\T_n)$ be a vertex such that for some $r$, $|(A_r(f-g))(v)| \geq \|f - g\|_{\dot{\CC}^1}$. For all points $w$ along the direction that $A_r$ acts, there is an upper bound on the value of $(f - g)(w)$ given by a quadratic whose second derivative is equal to $C s_2$ and whose slope  at $v$ has magnitude $\|f - g\|_{\dot{\CC}^1}$. But this means that this upper bound must at some point take a value less than $-c\left(\frac{\|f - g\|_{\dot{\CC}^1}}{s_2}\right)^2$. This implies that $c\left(\frac{\|f - g\|_{\dot{\CC}^1}}{s_2}\right)^2 \leq \eps_{0.5} n^2,$ leading to the desired bound.
\end{proof}

\section{Upper bounds on the volumes of covering polytopes}

\subsection{Polytopes used in the cover}\lab{sec:cover}
We will map $V(\T_n)$ onto $(\Z/n\Z) \times (\Z/n\Z)$ via the unique $\Z$ module isomorphism that maps $1$ to $(1, 0)$ and $\omega$ to $(0, 1)$. 
Without loss of generality (due linearity under scaling by a positive constant), we will assume in this and succeeding sections that that \beq\lim_{n \ra \infty}|P_n(s)|^{\frac{1}{n^2-1}} = 1.\eeq
Let $\eps_0$ be a fixed positive constant. Suppose  $x \in P_n(s)$ satisfies \beq\|x\|_2 > \eps_0 n^3.\lab{eq:eps0}\eeq


Given $n_1|n_2$, the natural map from $\Z^2$ to $\Z^2/(n_1 \Z^2) = V(\T_{n_1})$ factors through $\Z^2/(n_2 \Z^2) =V(\T_{n_2})$. We denote the respective resulting maps from $V(\T_{n_2})$ to $V(\T_{n_1}) $ by $\phi_{n_2, n_1}$, from $\Z^2$ to $V(\T_{n_2})$ by $\phi_{0, n_2}$ and from $\Z^2$ to $V(\T_{n_1})$ by $\phi_{0, n_1}$.
Given a set of boundary nodes $\b \subseteq V(\T_n)$, and $ x_\b \in \R^{\b}$, we define $Q_{ \b}(x)$ to be the fiber polytope over $x_\b$, that arises from the projection map $\Pi_{\b}$ of $P_n(s)$ onto $\R^{\b}.$ Note that $Q_{\b}(x)$ implicitly depends on $s$. 

Given positive $\eps_0,  \dots, \eps_k$  we will denote by $\eps_{k+1}$, a positive constant whose value may depend on the preceding $\eps_i$  but not on any $\eps_r$ for $r>k$. 
Let $o \in V(\T_n)$ be an offset that we will use to define $\b$. 
In this paper we will deal exclusively with the situation when $\eps_1^{-1} \in \Z$.
Let $n_2$ be the largest multiple of $ \eps_1^{-1} $ that  is less or equal to $n$. Thus, $n_2 = \eps_1^{-1} \lfloor n\eps_1 \rfloor.$ Note that $n_2 +  \eps_1^{-1}   \geq n$. Let $$n_1 = n_2\eps_1.$$ 
\begin{defn}\lab{defn:b}
We define the set $\b_{1} \subseteq V(\T_{n_1})$ of ``boundary vertices" to be all vertices that are either of the form $(0, y)$ or $(1, y)$ or $(x, 0)$ or $(x, 1)$, where $x, y$ range over all of $\Z/(n_1 \Z)$. We define the set $\b_{2} \subseteq V(\T_{n_2})$ to be $\phi_{n_2, n_1}^{-1}(\b_{1}).$
\end{defn}
Let $\rho_0:V(\T_{n_2}) \ra \{0, \dots, n_2-1\}^2\subseteq \Z^2$ be the unique map with this range that satisfies $\phi_{0, n_2} \circ \rho_0 = id$ on $V(\T_{n_2})$. We embed $V(\T_{n_2})$ into $V(\T_{n})$ via  
  $\phi_{0, n} \circ \rho_0,$ and define $$\tilde {\b} := \left(\phi_{0, n}\circ \rho_0(\b_2)\right)\cup \left(V(\T_{n})\setminus (\phi_{0, n}(\{0, \dots, n_2-1\}^2))\right).$$  Thus, we have the following.
\begin{defn}\lab{def:bhat}
 The set $\tilde{\b}$ is the union of the image of $\b_2$ under $\phi_{0, n}\circ \rho_0,$ with the set $\widehat\b$ of vertices that do not belong to  $\phi_{0, n} (\{0, \dots, n_2-1\}^2)$.
\end{defn}
Finally we define $\b$ to be $\tilde{\b} + o$, \ie a translation of $\tilde{\b}$ by the offset $o$.
Given $\b$, define $(x_\b)_{quant}$ to be the closest point to $x_\b$, every coordinate of which is an  integer multiple of $\frac{1}{M}$. 

\begin{defn}\lab{def:6.1} We define the polytope $\tilde{Q}_n(\b, s, x)$ as the preimage of $(x_\b)_{quant} + [-\frac{1}{M}, \frac{1}{M}]^\b$ under the coordinate projection $\Pi_\b$ of $P_n(s)$ onto $\R^\b$.
\end{defn}

\begin{lem}\lab{lem:polytope_number}
For sufficiently large $n$, the total number of distinct polytopes $\tilde{Q}_n(\b, s, x)$ as $x$ ranges over all points in $P_n(s)$ is at most   $(Cn^2M)^{(8 \eps_1^{-1})n + 2}.$
\end{lem}
\begin{proof}
The number of vertices in $\b$ is bounded above by $8 \eps_1^{-1}n$. Also, $x \in P_n(s)$ implies that $\|x\|_{\infty} < Cn^2$. The number of distinct points of the form $(x_\b)_{quant}$ can therefore be bounded above by $(Cn^2M)^{(8 \eps_1^{-1})n + 2}$ when $n$ is sufficiently large. Since the number of possible offsets is $n^2$, this places an upper bound of  $(Cn^2M)^{(8 \eps_1^{-1})n + 2}$ on the number of possible polytopes $\tilde{Q}_n(\b, s, x)$.
\end{proof}

In the remainder of this section, $s$ and $x$ and $\eps_1$ will be fixed, so the dependence of  various parameters on them will be suppressed.

\subsection{Bounding  $\left| (\tilde{Q}_n(\tilde{\b} + o) - x)\right|$ from above}\lab{ssec:appl_isop}

For $1 \leq i, j \leq \frac{n_2 }{n_1}$, and offset $o$, we define the $(i,j)^{th}$ square \beq\square_{ij}^o :=o + \phi_{0, n}\left(\left( \left[{(i-1)n_1} + 1, {in_1}\right]\times \left[{(j-1)n_1} + 1,  {jn_1}\right]\right)\cap \Z^2\right).\lab{eq:sqij}\eeq
We also define 
\beq\square^o :=o + \phi_{0, n}\left(\left( \left[ 1, n_2\right]\times \left[1,  n_2\right]\right)\cap \Z^2\right).\lab{eq:squaredef}\eeq
We note that the boundary vertices of each square 
$\square_{ij}^o$ are contained in $\b$. 
Let $\La_{ij}^o$ denote the orthogonal projection of $\R^{V(\T_n)}$ onto the subspace \beq A_{ij}^o := \left\{y\in \R^{\square_{ij}^o}\big|\sum_{k \in \square_{ij}^o}  y_k = 0\right\}.\eeq 
By abuse of notation, when necessary, we will identify the vertices in $\square_{ij}^o$ with the vertices in $V(\T_{n_1})$ in the natural way, and view $\La_{ij}^o$ as a projection of $\R^{V(\T_n)}$ onto the subspace \beq  \left\{y\in \R^{V(\T_{n_1})}\big|\sum_{k \in \square_{ij}^o}  y_k = 0\right\}.\eeq 
For any $z \in \tilde{Q}_n(\tilde{\b} + o) - x,$
the euclidean distance between $z$ and this subspace is less than $Cn^3$ by virtue of the upper bound of $Cn^2$ on the $l_\infty$ norm of $z$ and $x$. For sufficiently large $n$, we eliminate the $C$ and bound this euclidean distance from above by $n^4$. Therefore, for any fixed $o$, 
\beq n^{- \frac{4n_2^2}{n_1^2}}  \left|(\tilde{Q}_n(\tilde{\b} + o) - x)\right| & \leq &  \left|\prod\limits_{1 \leq i, j \leq \frac{n_2 }{n_1}} \La_{ij}^o (\tilde{Q}_n(\tilde{\b} + o) - x)\right| \\
& = &  \prod\limits_{1 \leq i, j \leq \frac{n_2 }{n_1}} \left| \La_{ij}^o (\tilde{Q}_n(\tilde{\b} + o) - x)\right|.\lab{eq:3.5} \eeq






We recall from (\ref{eq:2.2}) that
the anisotropic surface area of $L$ with respect to $K$, denoted $S_K(L)$, satisfies
\beqs S_K(L) \geq m|K|^{\frac{1}{m}} |L|^{\frac{m-1}{m}}.\eeqs 
Later, in Definition~\ref{def:tijo}, we will choose $t_{ij}^o$ carefully  depending on the restriction of $x$ to (a $1-$neighborhood of) $\square_{ij}^o$.
For $1 \leq i, j \leq \frac{n_2 }{n_1}$, let $P_{n_1}^{ij,o}( t_{ij}^o)$ be a copy of $ P_{n_1}( t_{ij}^o)$ in $\R^{\square_{ij}^o}$.
Taking $L_{ij}^o$ to be $ P_{n_1}^{ij,o}(t_{ij}^o)$  (note that $|P_{n_1}^{ij,o}(t_{ij}^o)| = |P_{n_1}(t_{ij}^o)|$), $K_{ij}^o$ to be $\La_{ij}^o (\tilde{Q}_n(\tilde{\b} + o) -x)$,  and $m_1 = n_1^2 -1$, this gives us 
$$m_1 |K_{ij}^o|^{\frac{1}{m_1}}|P_{n_1}(t_{ij}^o)|^{1 - \frac{1}{m_1}}  \leq {S_{K_{ij}^o}(L_{ij}^o)}.$$ 
We thus have 
\beq  |K_{ij}^o|^{\frac{1}{m_1}} \leq \frac{{S_{K_{ij}^o}(L_{ij}^o)}}{m_1 |P_{n_1}(t_{ij}^o)|^{1 - \frac{1}{m_1}} }.\eeq

Thus, 
\beqs \prod\limits_{1 \leq i, j \leq \frac{n_2 }{n_1}} \left(\left| \La_{ij}^o (\tilde{Q}_n(\tilde{\b} + o) - x)\right|  |P_{n_1}(t_{ij}^o)|^{m_1 - 1} \right) & \leq & 
 \prod\limits_{1 \leq i, j \leq \frac{n_2 }{n_1}} \left(  \frac{S_{K_{ij}^o}(L_{ij}^o)}{m_1}\right)^{m_1}.\eeqs

This, by (\ref{eq:3.5}) implies that 
 \beq \left( n^{- \frac{4n_2^2}{n_1^2}} \left|(\tilde{Q}_n(\tilde{\b} + o) - x)\right|\right) & \leq & 
 \prod\limits_{1 \leq i, j \leq \frac{n_2 }{n_1}} \frac{\left(\frac{S_{K_{ij}^o}(L_{ij}^o)}{m_1}\right)^{m_1}}{ |P_{n_1}(t_{ij}^o)|^{m_1 - 1}  }.\lab{eq:3.6}
\eeq

Recall from Subsection~\ref{sec:prelim} that for $a, b, c$ and $d$ the vertices of a lattice rhombus of side $1$ such that   $ a - d = -z\omega^2,$ $b-a = z,$ $c-b = -z \omega^2 ,$ $d-c = -z,$ for some $z \in \{1, \omega, \omega^2\}.$ In the respective cases when $z= 1, \omega$ or $\omega^2$, we define corresponding sets of lattice rhombi of side $1$ to be $E_0(\mathbb L)$, $E_1(\mathbb L)$ or $E_2(\mathbb L)$. This structure is carried over to $\T_n$ by the map $\phi_{0, n}$ defined in the beginning of Subsection~\ref{sec:cover}.
Recall from the beginning of Subsection~\ref{sec:cover} that we have mapped $V(\T_n)$ on to $(\Z/n\Z) \times (\Z/n\Z)$ by mapping $1$ to $(1, 0)$ and $\omega$ to $(0, 1)$ and extending this map to $V(\T_n)$ via a $\Z$ module homomorphism. In particular, this maps $1 + \omega$ to $(1, 1)$.

Let us examine $S_{K_{ij}^o}(L_{ij}^o)$ for a fixed $i, j$ and $o$. Note that $0 \in K_{ij}^o$. Let us identify $\square_{ij}^o$ with $V(\T_{n_1})$ labelled by $[1, n_1]^2\cap \Z^2$ by mapping the south west corner of $\square_{ij}^o$ onto $(1, 1)$.

\begin{defn}
 For $r\in\{0, 1, 2\}$ and $1 \leq k, \ell \leq n_1$, let $u^r_{k\ell}:= u^r_{k\ell}(i, j, o)$ denote the unit outward normal to the  facet of $L_{ij}^o$ that corresponds to the edge in $E_r(\T_{n_1})$, whose south west corner is $(k, \ell)$. 
Let $h_{k\ell}^r = h_{k\ell}^r(i, j, o)$ be the maximum value of the functional $\a(a) = \langle a, u^r_{k\ell} \rangle$ as $a$ ranges over $K_{ij}^o$. 
\end{defn}
Note that $K_{ij}^o$ does not depend on $t_{ij}^o$.
We see that \beq S_{K_{ij}^o}(L_{ij}^o) = \sum_{r \in \{0, 1, 2\}} w_r^{(n_1)}(t_{ij}^o)\left( \sum_{1 \leq k, \ell \leq n_1} h_{k\ell}^r(o, i, j)\right).\lab{eq:7.4}\eeq 

Now, for each $r \in \{0, 1, 2\}$, we define a linear map $D_r$ from $\R^{V(\T_{n'})}$ to $\R^{E_r(\T_{n'})},$ where $n'$ will be a positive integer made clear from context.
Let $f \in \R^{V(\T_{n'})}$ and $(v_1, v_2) \in V(\T_{n'})$. We use $e_r(v_1, v_2)$ to refer to an edge in $E_r(\T_n)$ whose south east corner is the vertex $(v_1, v_2)$. Then,
\ben
\item[(0)] $D_0 f(v_1-1, v_2-1) = \nabla^2 f(e_0(v_1-1, v_2-1)) =  -f(v_1, v_2-1) - f(v_1, v_2) + f(v_1-1, v_2-1) + f(v_1 + 1, v_2).$
\item[(1)] $D_1 f(v_1, v_2) =  \nabla^2 f(e_1(v_1, v_2)) = f(v_1+1, v_2) + f(v_1 , v_2 + 1) - f(v_1, v_2) - f(v_1+1, v_2 + 1).$
\item[(2)] $D_2 f(v_1-1, v_2-1) =  \nabla^2 f(e_2(v_1-1, v_2-1)) = -f(v_1, v_2) - f(v_1-1, v_2) + f(v_1, v_2+1) + f(v_1-1, v_2-1).$
\een

As stated earlier in (\ref{eq:A}), we also have the first order difference operators $A_0$, $A_1$ and $A_2$ given by 

\ben
\item[($\star 0$)] $A_0 f(v_1-1, v_2-1) =  - f(v_1-1, v_2-1) + f(v_1 - 1, v_2).$
\item[($\star 1$)] $A_1 f(v_1, v_2) =   - f(v_1 - 1, v_2 - 1) + f(v_1, v_2).$
\item[($\star 2$)] $A_2 f(v_1-1, v_2-1) =  - f(v_1 - 1, v_2 - 1) + f(v_1, v_2).$
\een

As a consequence, we see the following.
\beq D_2 = A_0A_1\lab{eq:3.10}\eeq
\beq D_0 = A_1A_2\lab{eq:3.11}\eeq 
\beq D_1 = - A_2A_0.\lab{eq:3.12}\eeq

Recall that $K_{ij}^o$ is $\La_{ij}^o (\tilde{Q}_n(\tilde{\b} + o) -x)$. For  linear maps $D_0, D_1$ and $D_2$ described above,  taking $n' = n_1$
we have for $1 \leq k, \ell \leq n_1-2$, and $r \in \{0, 1, 2\}$,
 \beq 0 \leq h_{k\ell}^r(o, i, j) = s_r - D_rx(o_1 + in_1 +k, o_2 + jn_1 + \ell).\lab{eq:7.5}\eeq
When either $k$ or $\ell$ is one of the numbers $n_1 -1$ or $n_1$, we see that  
$h_{k\ell}^r$ can be different due to the possibility of the constraints wrapping around. However, it is always true due to the quantization in Definition~\ref{def:6.1}, that
\beq  0 \leq h_{k\ell}^r(o, i, j) \leq \frac{4}{M} + s_r- D_rx(o_1 + in_1 + k, o_2 + jn_1  + \ell).\lab{eq:7.6}\eeq
Let \beq \tau_{k\ell}^r(o, i, j) := h_{k\ell}^r(o, i, j) - \left( s_r- D_rx(o_1 + in_1 + k, o_2 + jn_1  + \ell)\right).\eeq
Thus, \beq D_rx(o_1 + in_1 + k, o_2 + jn_1  + \ell) - s_r \leq \tau_{k\ell}^r(o, i, j)  \leq \frac{4}{M}.\lab{eq:3.16}\eeq

\begin{figure}\label{fig:subsquare}
\begin{center}
\includegraphics[scale=0.60]{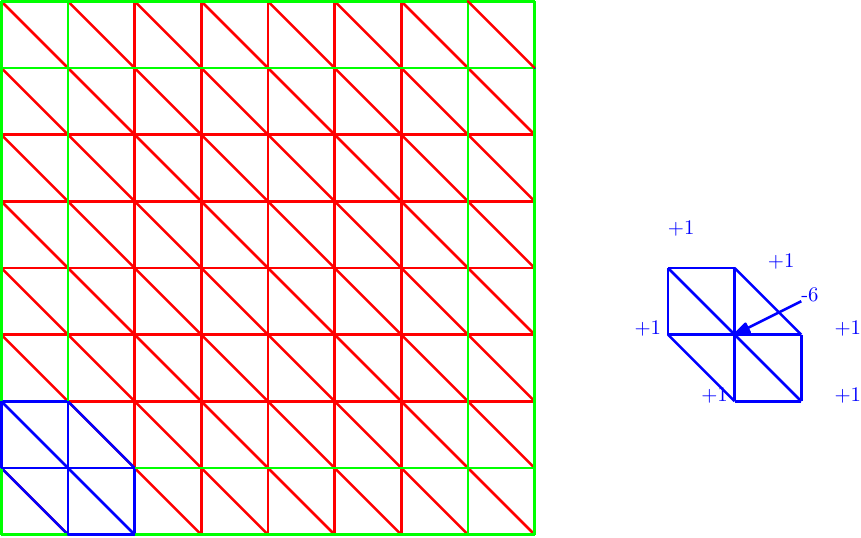}
\caption{The $(2,2,2)-$Laplacian acting on functions defined on a subsquare $\square_{ij}^o$
}
\end{center}
\end{figure}
 
Denoting $t_{ij}^o$ by $t$,   we let  $\Delta$ be defined by the following.

 \beqs 2(\Delta \ast y)(v) & := & w_0^{(n_1)}(t)(D_0y(v_1 -1, v_2 -1) + D_0y(v_1 -1, v_2))\nonumber\\
& + & w_1^{(n_1)}(t)(D_1y(v_1, v_2 ) + D_1y(v_1 -1, v_2-1))\nonumber\\
& + & w_2^{(n_1)}(t)(D_2y(v_1 -1, v_2 -1) + D_2y(v_1, v_2-1)).\eeqs
A rearrangement of this gives us 
\beq 2(\Delta \ast y)(i, j) & = & (- \wone_0 + \wone_1 + \wone_2) (y(i, j+1) - 2 y(i, j) + y(i, j-1))\nonumber\\
                                            & + & ( \wone_0 - \wone_1 + \wone_2) (y(i+1, j+1) - 2 y(i, j) + y(i-1, j-1))\nonumber\\
					& + & (\wone_0 + \wone_1 - \wone_2) (y(i+1, j) - 2 y(i, j) + y(i-1, j)).\lab{eq:rearrange}
\eeq
We note that apart from $n$ in (\ref{eq:7.7}) being substituted by $n_1$ in (\ref{eq:rearrange}) there is also a multiplicating rescaling by $|P_n(s)|$, while going from $\De_w$ in (\ref{eq:7.7}) to $\De$ in 
(\ref{eq:rearrange}).
We will now remark on $\frac{S_{K_{ij}^o}(L_{ij}^o)}{m}$. 
\beqs S_{K_{ij}^o}(L_{ij}^o) = \sum_{r \in \{0, 1, 2\}} w_r^{(n_1)}(t_{ij}^o)\left( \sum_{1 \leq k, \ell \leq n_1} h_{k\ell}^r(o, i, j)\right).\eeqs


We observe that \beq \left(\frac{S_{K_{ij}^o}(L_{ij}^o)}{m_1}\right) & = & \sum_{r \in \{0, 1, 2\}} \frac{w_r^{(n_1)}(t)}{m_1}\left( \sum_{1 \leq k, \ell \leq n_1} h_{k\ell}^r(i, j, o)\right)\nonumber\\
& =  &  \sum_{r \in \{0, 1, 2\}} \frac{w_r^{(n_1)}(t)}{m_1}\left( \sum_{n_1 - 1 \leq k, \ell \leq n_1} \tau_{k\ell}^r(i, j, o)\right)  \nonumber\\& + &   \frac{\sum_{r \in \{0, 1, 2\}} n_1^2 {s}_r w_r^{(n_1)}(t_{ij}^o)}{m_1}\nonumber\\ &-&   m_1^{-1}\sum_{1 \leq k, \ell \leq n_1}  \left( \Delta_{t_{ij}^o} \ast x + \de_{t_{ij}^o} (x)\right)\left(o_1 + in_1 + k, o_2 + jn_1 + \ell\right)\nonumber\\\lab{eq:le}.\eeq
Here $\de_{t_{ij}^o} (x)$ comprises of a sum of some linear terms in $w_r^{(n_1)}(t_{ij}^o)A_0x$, $w_r^{(n_1)}(t_{ij}^o)A_1x$ and $w_r^{(n_1)}(t_{ij}^o)A_2x$ which are nonzero only  on ${k, \ell} (\mathrm{mod}\, n_1) \in \{ -1, 0, 1\}$, in addition to a linear term depending on $M^{-1}$. These terms correct for the behavior of $ \Delta_{t_{ij}^o} \ast x$ at the boundary of $\square_{ij}^o$.


\begin{defn}\lab{def:h0}
Let $\eps_1:= (h_0k_0)^{-1}$, where $h_0$ is some large positive integer that will be chosen later as a function of $n$. 
\end{defn}
We will choose the offset $o$ to be $0$, and consider the squares
$\square_{ij}^o$ as in (\ref{eq:squaredef}).

Let us define $q(t)$ as in Lemma~\ref{lem:2.3} to be the unique quadratic polynomial from $\L$ to $\R$ such that $\nabla^2q$ satisfies the following.

\ben 
\item  $\nabla^2q(e) = - t_0,$ if $e \in E_0(\mathbb L)$.
\item $\nabla^2q(e)  =  - t_1,$  if $e \in E_1(\mathbb L )$.
\item $\nabla^2q(e)  =  - t_2,$ if  $e \in E_2(\mathbb L)$.
\item $q((0, 0)) = q((n, 0)) = q((0, n)) = 0$.
\een
Let $f \in B_\infty(g, \eps_{0.5}n^2)$.
Recall from (\ref{eq:7.7}) that 
 $\Delta = \Delta_{t_{ij}^o}$ is the function from $V(\T_{n})$ to $\R$, uniquely specified by the following condition. For any $f:V(\T_{n}) \ra \R$, and $(v_1, v_2) = v \in V(\T_{n})$,
\beqs 2(\Delta \ast f)(v) & = & w_0^{(n_1)}(t_{ij}^o)(D_0f(v_1 -1, v_2 -1) + D_0f(v_1 -1, v_2))\nonumber\\
& + & w_1^{(n_1)}(t_{ij}^o)(D_1f(v_1, v_2 ) + D_1f(v_1 -1, v_2-1))\nonumber\\
& + & w_2^{(n_1)}(t_{ij}^o)(D_2f(v_1 -1, v_2 -1) + D_2f(v_1, v_2-1)).\eeqs

Let $\Phi$ be the function from $V(\T_n)$ to $\R$, given by \beq \Phi := \frac{\II(\square_{11}^0)}{n_1^2}, \eeq where for a subset $S$ of $V(\T_n)$, $\II(S)$ is the indicator function of $S$. 
\begin{defn}
We set $\tilde{t}_{ij}^o= \tilde{t}_{ij}^o(g)$ to that unique value of $t$ such that   
\beq   D_0 (\Phi \ast g)\left(o_1 + in_1 + k - 1, o_2 + jn_1 + \ell\right)  + \nonumber\\  D_0 (\Phi \ast g)\left(o_1 + in_1 + k - 1, o_2 + jn_1 + \ell - 1\right)   = 2(t_0 - s_0)\lab{eq:3.17}.\eeq
\beq   D_1 (\Phi \ast g)\left(o_1 + in_1 + k, o_2 + jn_1 + \ell\right)  +  \nonumber\\   D_0 (\Phi \ast g)\left(o_1 + in_1 + k - 1, o_2 + jn_1 + \ell - 1\right)   = 2(t_1 - s_1).\eeq
\beq   D_2 (\Phi \ast g)\left(o_1 + in_1 + k - 1, o_2 + jn_1 + \ell-1\right)  +  \nonumber\\   D_1 (\Phi \ast g)\left(o_1 + in_1 + k, o_2 + jn_1 + \ell -1 \right)   = 2(t_2 - s_2).\eeq
\end{defn}

\begin{defn}\lab{def:tijo}
We set ${t}_{ij}^o = t_{ij}^o(f)$ to that unique value of $t$ such that for each $r \in \{0, 1, 2\}$,
\beq \sum_{1 \leq k, \ell \leq n_1}  \left( D_r (f ) - \tau_{k\ell}^r\right)\left(o_1 + in_1 + k, o_2 + jn_1 + \ell\right)  = 2(t - s).\lab{eq:3.18}\eeq
\end{defn}

For this value of $t$, (\ref{eq:le}) gives us 
\beq \left(\frac{S_{K_{ij}^o}(L_{ij}^o)}{m_1}\right) & = & \sum_{r \in \{0, 1, 2\}} \frac{w_r^{(n_1)}(t)}{m_1}\left( \sum_{1 \leq k, \ell \leq n_1} h_{k\ell}^r(i, j, o)\right)\\
& = &   \frac{\sum_{r \in \{0, 1, 2\}} n_1^2 {t}_r w_r^{(n_1)}(t_{ij}^o)}{m_1}\nonumber\lab{eq:le1}
\eeq
In the above expression, by $t_r$, we mean $({t}_{ij}^o)_r$.

Suppose  $g = \Re \hat g$,  where $\hat g$ is the associated scaled complex exponential.
Since $h_0 k_0 n_1 \leq n$, and our choice of $h_0$ is $\omega(1)$, we see that  $\Phi \ast \hat g = \la \hat g$ for some  complex number $\la$, such that $1 \geq |\la| \geq 1 - o(1),$ and $\arg(\la) = o(1).$ 
Therefore, for each $1 \leq i, j \leq \frac{n_2}{n_1},$
\beqs \Re \left(\la D_r g\left(o_1 + in_1, o_2 + jn_1\right)\right)  = \left((\tilde{t}_{ij}^o)_r - s_r\right).\eeqs Thus,
 \beq (\tilde{t}_{ij}^o)_0 = s_0 + \Re((\omega_n^{k_0} - 1)(1 - \omega_n^{-(k_0 + \ell_0)}))\Re(\la \hat g)\left(o_1 + in_1, o_2 + jn_1\right),\eeq
\beq (\tilde{t}_{ij}^o)_1 = s_1 + \Re(-(\omega_n^{k_0} - 1)(\omega_n^{\ell_0} - 1))\Re(\la \hat  g)\left(o_1 + in_1, o_2 + jn_1\right),\eeq
and
\beq (\tilde{t}_{ij}^o)_2 = s_2 + \Re((\omega_n^{\ell_0} - 1)(1 - \omega_n^{-(k_0 + \ell_0)} ))\Re(\la \hat  g)\left(o_1 + in_1, o_2 + jn_1\right).\eeq
In particular, we see that for any $r, r'$, the ratio
\beq \frac{(\tilde{t}_{ij}^o)_r - s_r}{(\tilde{t}_{ij}^o)_{r'} - s_{r'}}\lab{eq:3.28} \eeq is independent of $o, i,j$, whenever the denominator is nonzero; otherwise the numerator is zero as well.

\begin{lem}\lab{lem:3.4} Let $f \in B_\infty(g, \eps_{0.5}n^2)$ , and suppose that $f$ satisfies 
$$\|D_r( f)\|_\infty < \check{C} s_r \log n,$$ for $r \in \{0, 1, 2\}$. Then, 
the following estimate for $|t_{ij}^o(f) - \tilde{t}_{ij}^o(g)|$ holds.

\beqs |t_{ij}^o(f) - \tilde{t}_{ij}^o(g)| <  Cn_1^{-1} \left(\|f - g\|_{\dot{\CC}^1}  + \check{C} s_2 \log n + M^{-1}\right),  \eeqs
\end{lem}
\begin{proof}

Using (\ref{eq:3.10}), (\ref{eq:3.11}) and (\ref{eq:3.12}) 
  it follows that each of (\ref{eq:3.17}) to (\ref{eq:3.18}) collapses as a telescoping sum. We proceed to elaborate on the case of $r= 0$ in some detail. The cases of $r= 1$ and $r = 2$ and analogous.
Suppose $r = 0$. In the present situation $o = 0$. Let $\bar{f}$ and $\bar{g}$ denote periodic functions on $\L$ with whose representatives in $P_n(s)$ are respectively $f$ and $g$. Then,
\beqs \sum_{1 \leq k, \ell \leq n_1}  \left( D_r (\bar g  + q(s) - q(\tilde{t}_{ij}^o))\right)\left( in_1 + k,  jn_1 + \ell\right)  = 0,\eeqs and
\beqs \sum_{1 \leq k, \ell \leq n_1}  \left( D_r (\bar{f}  + q(s) - q(t_{ij}^o)) - \tau_{k\ell}^r\right)\left( in_1 + k,  jn_1 + \ell\right)  = 0,\eeqs
together give us 

\beqs \sum_{1 \leq k, \ell \leq n_1}  \left( D_r (\bar{f} - \bar{g}   - q(t_{ij}^o - \tilde{t}_{ij}^o) ) - \tau_{k\ell}^r\right)
\left( in_1 + k,  jn_1 + \ell\right)  = 0.\eeqs

This implies that 

\beqs \sum_{1 \leq k, \ell \leq n_1} \left( \left( D_r (\bar{f} - \bar{g}    )  \right)\left( in_1 + k,  jn_1 + \ell\right) - ((t_{ij}^o)_r - (\tilde{t}_{ij}^o)_r)\right) = \\
 \sum_{1 \leq k, \ell \leq n_1}   \tau_{k\ell}^r \left( in_1 + k,  jn_1 + \ell\right) . \eeqs
It follows that \beq ((t_{ij}^o)_r - (\tilde{t}_{ij}^o)_r) = n_1^{-2}\sum_{1 \leq k, \ell \leq n_1}  \left(\left( D_r (\bar{f} - \bar{g}    )- \tau_{k\ell}^r\right) \left( in_1 + k,  jn_1 + \ell\right)\right).\lab{eq:t}\eeq

If $r =0$, $D_r = A_{2}A_1. $

In this case, 
\beqs \sum_{1 \leq k, \ell \leq n_1}   \left(\left( D_0 (\bar{f} - \bar{g}    )\right) \left( in_1 + k,  jn_1 + \ell\right)- \tau_{k\ell}^r\left( in_1 + k,  jn_1 + \ell\right)\right)  =  \\   \sum_{1 \leq k, \ell \leq n_1}   \left(\left( A_{2} A_1 (\bar{f} - \bar{g}    )\right) \left( in_1 + k,  jn_1 + \ell\right)- \tau_{k\ell}^r\left( in_1 + k,  jn_1 + \ell\right)\right)  =   \\
  \left( \sum_{1 \leq \ell \leq n_1}  \left( A_1 (\bar{f} - \bar{g}    )\right) \left( (i+1)n_1,  jn_1 + \ell\right)-   \sum_{1 \leq k, \ell \leq n_1} \tau_{k\ell}^r\left( (i+1)n_1,  jn_1 + \ell\right)\right)  -   \\
\left( \sum_{1 \leq \ell \leq n_1}  \left( A_1 (\bar{f} - \bar{g}    )\right) \left(in_1 + 1,  jn_1 + \ell\right)-  \sum_{1 \leq k, \ell \leq n_1} \tau_{k\ell}^r\left( in_1+1,  jn_1 + \ell\right)\right).
\eeqs

By Lemma~\ref{lem:2.125} and (\ref{eq:3.16}) we see that 
\beqs \left|\sum_{1 \leq k, \ell \leq n_1}   \left(\left( D_0 (\bar{f} - \bar{g}    )\right) \left( in_1 + k,  jn_1 + \ell\right)- \tau_{k\ell}^r\left( in_1 + k,  jn_1 + \ell\right)\right) \right| <\\ Cn_1 \left(\|f - g\|_{\dot{\CC}^1}  + \check{C} s_0 \log n + M^{-1}\right).\eeqs
Analogous computations give us

\beqs \left|\sum_{1 \leq k, \ell \leq n_1}   \left(\left( D_1 (\bar{f} - \bar{g}    )\right) \left( in_1 + k,  jn_1 + \ell\right)- \tau_{k\ell}^r\left( in_1 + k,  jn_1 + \ell\right)\right) \right| <\\ Cn_1 \left(\|f - g\|_{\dot{\CC}^1}  + \check{C} s_1 \log n + M^{-1}\right),\eeqs
and 

\beqs \left|\sum_{1 \leq k, \ell \leq n_1}   \left(\left( D_2 (\bar{f} - \bar{g}    )\right) \left( in_1 + k,  jn_1 + \ell\right)- \tau_{k\ell}^r\left( in_1 + k,  jn_1 + \ell\right)\right) \right| <\\ Cn_1 \left(\|f - g\|_{\dot{\CC}^1}  + \check{C} s_2 \log n + M^{-1}\right).\eeqs
Together, the last three equations give us 
\beqs |t_{ij}^o - \tilde{t}_{ij}^o| <  Cn_1^{-1} \left(\|f - g\|_{\dot{\CC}^1}  + \check{C} s_2 \log n + M^{-1}\right),  \eeqs proving Lemma~\ref{lem:3.4}.
\end{proof}


Thus (\ref{eq:le1}) gives us 

\beq \left(\frac{S_{K_{ij}^o}(L_{ij}^o)}{m_1}\right) 
& =&     \frac{\sum_{r \in \{0, 1, 2\}} n_1^2 {t}_r w_r^{(n_1)}(t_{ij}^o)}{m_1}\nonumber\\
& = &  |P_{n_1}(t_{ij}^o)|.\lab{eq:le3}
\eeq
\begin{lem}
\lab{lem:3.7}
\beq \left|(\tilde{Q}_n(\tilde{\b} + o) - f)\right| & \leq & n^{\frac{4n_2^2}{n_1^2}}\prod\limits_{1 \leq i, j \leq \frac{n_2 }{n_1}}   |P_{n_1}(t_{ij}^o)|.
\eeq

\end{lem}
\begin{proof}
Recall by  (\ref{eq:3.6}), that

 \beqs \left( n^{- \frac{4n_2^2}{n_1^2}} \left|(\tilde{Q}_n(\tilde{\b} + o) - x)\right|\right) & \leq & 
 \prod\limits_{1 \leq i, j \leq \frac{n_2 }{n_1}} \frac{\left(\frac{S_{K_{ij}^o}(L_{ij}^o)}{m_1}\right)^{m_1}}{ |P_{n_1}(t_{ij}^o)|^{m_1 - 1}  }.
\eeqs
Therefore,
 \beq \left( n^{- \frac{4n_2^2}{n_1^2}} \left|(\tilde{Q}_n(\tilde{\b} + o) - f)\right|\right)^\frac{1}{m_1} & \leq & 
 \prod\limits_{1 \leq i, j \leq \frac{n_2 }{n_1}} \frac{\left(\frac{S_{K_{ij}^o}(L_{ij}^o)}{m_1}\right)}{ |P_{n_1}(t_{ij}^o)|^{1 - \frac{1}{m_1}}  }\\ 
& = & \prod\limits_{1 \leq i, j \leq \frac{n_2 }{n_1}}  |P_{n_1}(t_{ij}^o)|^{\frac{1}{m_1}}.\nonumber
\eeq
Thus,
\beqs \left|(\tilde{Q}_n(\tilde{\b} + o) - f)\right| & \leq & n^{\frac{4n_2^2}{n_1^2}}\prod\limits_{1 \leq i, j \leq \frac{n_2 }{n_1}}   |P_{n_1}(t_{ij}^o)|.
\eeqs
\end{proof}

\section{Concentration of measure in $P_n(s)$ with respect to $\ell_\infty$.}

Let $\sigma(s) := - \log \f(s)$ be called the surface tension at $s$.
\subsection{Concentration when the surface tension is strictly convex at $s$.}

For an arbitrary concave function $\kappa$ of $\R^3_+$, we will use $\nabla\kappa(x)$  to denote a supergradient, that is, some vector $v$ such that 
$$\kappa(y) - \kappa(x) \leq v \cdot (y-x),$$ for all $y$ in the domain of $\kappa.$ 
We assume that the surface tension $\sigma = - \log {\mathbf f}$ is strictly convex at $s$. 
The strict convexity of the surface tension  at $s$ implies the following: For any $t_{ij}^o \neq s$,  and any choice of supergradient $\nabla f(s)$ (since as far as we know, this need not be unique)
\beqs \log \mathbf{f}(\tilde{t}_{ij}^o(g)) < \log \mathbf{f}(s) + (\tilde{t}_{ij}^o(g) - s) \cdot \left(\frac{\nabla \mathbf{f}(s)}{\mathbf{f}(s)}\right). \eeqs 

Under this condition, we shall show that if  $f$ is sampled from $P_n(s)$ randomly, then for any fixed $\eps_0 > 0$,
\beq \lim_{n \ra 0} \p\left[\|f\|_\infty > \eps_0 n^2\right] = 0.\lab{eq:conc}\eeq

\begin{defn}
Let the defect $ \log \mathbf{f}(t) - \log \mathbf{f}(s) -  (t- s) \cdot \left(\frac{\nabla \mathbf{f}(s)}{\mathbf{f}(s)}\right)$ be denoted by $\mathbf{dfc}(t, s).$ 
\end{defn}
Note that for any $s$ where $\mathbf{f}$ is $\CC^1$, and any $t \neq s$, this defect is strictly negative due to the assumption of strict concavity of entropy.

By Corollary~\ref{cor:2.5} and Corollary~\ref{cor:...} we have 
\beqs   \left(1 - \frac{C\log n_1}{n_1}\right) \mathbf{f}(s) \leq \mathbf{f}_{n_1}(s) \leq \left(1 + \frac{C\log n_1}{n_1}\right) \mathbf{f}(s). \eeqs

Bronshtein \cite{Brsh} (see also \cite{Wellner}) obtained an upper bound of  \beq C_{br}\eps^{-\frac{d}{2}} \lab{eq:Br}\eeq
for the logarithm of the $L^\infty$ covering number
of the class of all convex functions $g$ defined on a fixed convex body $\Omega$ in $\R^d$
satisfying a (uniform) Lipschitz condition: $|g(y) - g(x)| \leq L|y - x|$ for all
$x, y \in \Omega$.
We note that the functions in $P_n(s)$ are $O(n)$-Lipschitz when extended to the continuous torus in the natural piecewise linear fashion.

We shall now set some parameters. 

\begin{defn}\lab{def:3.8}
Let $h_0$ (see Definition~\ref{def:h0}) be set so that $ n_1 \sim \eps_1 n$. 
Let $M = (s_2 \sqrt{\eps_{0.5}}  n )^{-1}$ and $\check{C} = \frac{\sqrt{\eps_{0.5}}n}{\log n}.$ 
\end{defn}


We are now ready to prove the following theorem.

\begin{thm} \lab{thm:1} Let $s$ be point in $\R_+^3$ such that the surface tension $\sigma(s) = - \log \mathbf{f}(s)$ is strictly convex at $s$. Let $\eps_0$ be a universal constant greater than $0$. Then, for any positive $\de$,  for all sufficiently large $n$, 
\beq |P_n(s)\setminus B_\infty\left(0, \eps_0 n^2\right)| & \leq & \de|P_n(s)|.\eeq
\end{thm}
\begin{proof}
By Lemma~\ref{lem:2.13},
\beqs \|f - g\|_{\dot{\CC}^1} \leq Cs_2 \sqrt{\eps_{0.5}} n.\eeqs

Also, by Lemma~\ref{lem:3.4}, we see that
if $f$ satisfies 
$$\|D_r( f)\|_\infty < \check{C} s_r \log n,$$ for $r \in \{0, 1, 2\}$ then, 

\beq|t_{ij}^o(f) - \tilde{t}_{ij}^o(g)| & < & Cn_1^{-1} \left(\|f - g\|_{\dot{\CC}^1}  + \check{C} s_2 \log n + M^{-1}\right)\nonumber\\
                                                            & < &  Cn_1^{-1} \left( Cs_2 \sqrt{\eps_{0.5}} n  + \check{C} s_2 \log n + M^{-1}\right)\nonumber\\
& < & C \left(\frac{\sqrt{\eps_{0.5}}}{\eps_1 }\right).\lab{eq:comp}\eeq

Recall from Lemma~\ref{lem:3.7} that :

\beqs \left|(\tilde{Q}_n(\tilde{\b} + o) - f)\right| & \leq & n^{\frac{4n_2^2}{n_1^2}}\prod\limits_{1 \leq i, j \leq \frac{n_2 }{n_1}}   |P_{n_1}(t_{ij}^o)|.
\eeqs
In view of (\ref{eq:comp}) and Corollary~\ref{cor:lip} (which states that  $|\mathbf{f}_n(s) - \mathbf{f}_n(t)| < (2e + \eps)|s - t|.$), we therefore have

\beq \left|(\tilde{Q}_n(\tilde{\b} + o) - f)\right| & \leq & n^{\frac{4n_2^2}{n_1^2}}\exp \sum\limits_{1 \leq i, j \leq \frac{n_2 }{n_1}}(n_1^2 - 1)\ln |\mathbf{f}_{n_1}(t_{ij}^o(f))|\\
& \leq & n^{\frac{4n_2^2}{n_1^2}}\exp \sum\limits_{1 \leq i, j \leq \frac{n_2 }{n_1}}(n_1^2 - 1)\ln \left(\left|\mathbf{f}_{n_1}(\tilde{t}_{ij}^o(g)) + C\left(\frac{\sqrt{\eps_{0.5}}}{\eps_1 }\right)\right|\right)\\
& \leq & n^{\frac{4n_2^2}{n_1^2}}\exp \sum\limits_{1 \leq i, j \leq \frac{n_2 }{n_1}}(n_1^2 - 1)\left(\ln \left(\left|\mathbf{f}(\tilde{t}_{ij}^o(g))\right) + C\left(\frac{\log n_1}{n_1}\right) + C\left(\frac{\sqrt{\eps_{0.5}}}{\eps_1 }\right)\right|\right).\nonumber\\
\lab{eq:from-lem-3.6}
\eeq
 
We simplify \beqs \ln \left(\left|\mathbf{f}(\tilde{t}_{ij}^o(g)) +  C\left(\frac{\log n_1}{n_1}\right) + C\left(\frac{\sqrt{\eps_{0.5}}}{\eps_1 }\right)\right|\right)\eeqs
further as follows.
\beqs
\log \left(\left|\mathbf{f}(\tilde{t}_{ij}^o(g)) +  C\left(\frac{\log n_1}{n_1}\right) + C\left(\frac{\sqrt{\eps_{0.5}}}{\eps_1 }\right)\right|\right) & \leq & 
\ln \left(\left|\mathbf{f}(\tilde{t}_{ij}^o(g))\right|\right)  + C\left(\frac{\sqrt{\eps_{0.5}}}{\eps_1 }\right).\eeqs
This is in turn less or equal to \beqs \log \mathbf{f}(s) + (\tilde{t}_{ij}^o(g) - s) \cdot \left(\frac{\nabla \mathbf{f}(s)}{\mathbf{f}(s)}\right) + \mathbf{dfc}(\tilde{t}_{ij}^o(g), s) + C\left(\frac{\sqrt{\eps_{0.5}}}{\eps_1 }\right).
\eeqs

Thus,

\beqs \sum\limits_{1 \leq i, j \leq \frac{n_2 }{n_1}}(n_1^2 - 1)\ln \left(\left|\mathbf{f}(\tilde{t}_{ij}^o(g)) + C\left(\frac{\log n_1}{n_1}\right) + C\left(\frac{\sqrt{\eps_{0.5}}}{\eps_1 }\right)\right|\right) \eeqs is less or equal to \beq
\sum\limits_{1 \leq i, j \leq \frac{n_2 }{n_1}}n_1^2\left(\log \left(\mathbf{f}(s)\right) +  \mathbf{dfc}(\tilde{t}_{ij}^o(g), s) + C\left(\frac{\sqrt{\eps_{0.5}}}{\eps_1 }\right)\right).
\lab{eq:last?}\eeq
Recall that $\mathbf{dfc}(\tilde{t}_{ij}^o(g), s)$ is strictly negative. Thus,
By setting $\frac{\sqrt{\eps_{0.5}}}{\eps_1}$ to be a sufficiently small universal constant, we can ensure that (\ref{eq:last?}) 
is less or equal to 
\beqs
\sum\limits_{1 \leq i, j \leq \frac{n_2 }{n_1}}n_1^2\left(\log \left(\mathbf{f}(s)\right) +  \frac{\mathbf{dfc}(\tilde{t}_{ij}^o(g), s)}{2}\right)
\eeqs
for any $\tilde{t}$ corresponding to eigenfunction indices $k_0, \ell_0$ generated from Lemma~\ref{lem:2.11}. 

As a result of this, we see by (\ref{eq:from-lem-3.6}) that
\beqs \left|(\tilde{Q}_n(\tilde{\b} + o) - f)\right| & \leq & 
 n^{\frac{4n_2^2}{n_1^2}}\exp  \sum\limits_{1 \leq i, j \leq \frac{n_2 }{n_1}}n_1^2\left(\log \left(\mathbf{f}(s)\right) +  \frac{\mathbf{dfc}(\tilde{t}_{ij}^o(g), s)}{2}\right).\eeqs
It follows that
\beq \left|(\tilde{Q}_n(\tilde{\b} + o) - f)\right|  \mathbf{f}(s)^{-n_2^2} & \leq & 
 n^{\frac{4n_2^2}{n_1^2}}\exp  \sum\limits_{1 \leq i, j \leq \frac{n_2 }{n_1}}n_1^2\left(  \frac{\mathbf{dfc}(\tilde{t}_{ij}^o(g), s)}{2}\right).\lab{eq:end...}\eeq
As a result,

\beqs \left|(\tilde{Q}_n(\tilde{\b} + o) - f)\right| |P_n(s)|^{-1} & \leq &  \exp(- \eps_2 n^2),\eeqs
where $\eps_2$ is some universal constant greater than $0$, depending on $\eps_0$ and $s$ alone.

Let $\A_1$ be the subset of $P_n(s)$ consisting of all those $f$ for which $\|f\|_{L_2^2} \geq \check{C}_2n\ln \de^{-1}$, which by  Lemma~\ref{lem:2.11}, has measure at most $\de|P_n(s)|$. 
Let $\A_2$ be the subset of $P_n(s)$ consisting of all those $f$ such that for each $r\in \{0, 1, 2\}$, $\|D_r(f)\|_\infty > \check{C}\log n$, which by Lemma~\ref{lem:2.125} has measure at most $n^{-c\check{C}+2}|P_n(s)|$.

Using Lemma~\ref{lem:polytope_number}, which provides an upper bound on the number of covering polytopes $\tilde{Q}(\b, s, x)$, we have the following for all sufficiently large $n$.
\beqs \left|\left(P_n(s) \setminus (\A_1\cup \A_2)\right) \cap B_\infty(g, \eps_{0.5}n^2)\right| < \exp(- \frac{\eps_2 n^2}{2})|P_n(s)|.\eeqs
Using Bronshtein's upper bound of (\ref{eq:Br})  and Lemma~\ref{lem:2.14} with $\check{C} = P_n(s)\setminus(\A_1 \cup \A_2)$, we see that 
\beqs \left|P_n(s) \setminus (\A_1\cup \A_2 \cup B_\infty(0, \eps_0n^2)) \right| <\exp( C_{br}\eps_{0.5}^{-1})\exp(- \frac{\eps_2 n^2}{2})|P_n(s)|.\eeqs

Therefore
\beq |P_n(s)\setminus B_\infty\left(0, \eps_0 n^2\right)| & \leq & \left(\left(n^{-c\check{C} + 2}\right) + \exp(- \frac{\eps_2}{2} n^2) + \de\right)|P_n(s)|.\eeq
This completes our proof.
\end{proof}


\subsection{Concentration when a subgradient of the surface tension belongs to a certain cone.}

As stated earlier, Bronshtein \cite{Brsh} (see also \cite{Wellner}) obtained an upper bound of  $C\eps^{-\frac{d}{2}}$
for the logarithm of the $L^\infty$ covering number at scale $\eps$
of the class of all convex functions $g$ defined on a fixed convex body $\Omega$ in $\R^d$
satisfying a (uniform) Lipschitz condition: $|g(y) - g(x)| \leq L|y - x|$ for all
$x, y \in \Omega$.
We note that the functions in $P_n(s)$ are $O(n)$-Lipschitz when extended to the continuous torus in the natural piecewise linear fashion.

We shall now set some parameters. 

\begin{defn}\lab{def:3.8}
Let $h_0$ (see Definition~\ref{def:h0}) be set so that $ n_1 \sim \eps_1 n$. 
Let $M = (s_2 \sqrt{\eps_{0.5}}  n )^{-1}$ and $\check{C} = \frac{\sqrt{\eps_{0.5}}n}{\log n}.$ 
\end{defn}

\begin{thm}\lab{thm:2}
Let $w$ be a supergradient of $\f$ at $s$.
Suppose that 
\beq w_0^2 + w_1^2 + w_2^2 < 2\left(w_0 w_1 + w_1 w_2 + w_2 w_0\right).\eeq
 Then, for any $\eps > 0$, 
we have  $$\lim_{n \ra 0} \p\left[\|g\|_\infty > n^{\frac{7}{4} + \eps}\right] = 0$$  when $g$ is randomly sampled from the uniform measure on $P_n(s)$.
\end{thm}
\begin{proof}
Suppose $g \in P_n(s)$ and $\|g\|_2 \geq \eps_0 n^3.$
Recall from Lemma~\ref{lem:2.11} that $k_0, \ell_0$ are bounded above in magnitude by $\frac{\log \de^{-1}}{c\eps_0^2}$. 
Therefore, 
 \beqs (\tilde{t}_{ij}^o)_0 & = & s_0 + \Re((\omega_n^{k_0} - 1)(1 - \omega_n^{-(k_0 + \ell_0)}))\Re(\la \hat g)\left(o_1 + in_1, o_2 + jn_1\right)\\
& = & s_0 - \frac{4\pi^2k_0(k_0 + \ell_0)}{n^2}(1 + O(\max(|k_0|, |\ell_0|) n^{-1}))\Re(\la \hat g)\left(o_1 + in_1, o_2 + jn_1\right).\eeqs
\beqs (\tilde{t}_{ij}^o)_1 & = & s_1 + \Re(-(\omega_n^{k_0} - 1)(\omega_n^{\ell_0} - 1))\Re(\la \hat  g)\left(o_1 + in_1, o_2 + jn_1\right)\\
& = &  s_1 + \frac{4\pi^2k_0\ell_0}{n^2}(1 + O(\max(|k_0|, |\ell_0|) n^{-1}))\Re(\la \hat g)\left(o_1 + in_1, o_2 + jn_1\right).\eeqs
and
\beqs (\tilde{t}_{ij}^o)_2 & = & s_2 + \Re((\omega_n^{\ell_0} - 1)(1 - \omega_n^{-(k_0 + \ell_0)} ))\Re(\la \hat  g)\left(o_1 + in_1, o_2 + jn_1\right)\\
& = & s_2 - \frac{4\pi^2\ell_0(k_0 + \ell_0)}{n^2}(1 + O(\max(|k_0|, |\ell_0|)n^{-1}))\Re(\la \hat g)\left(o_1 + in_1, o_2 + jn_1\right).\eeqs

Recall from (\ref{eq:3.28}) that 
for any $r, r'$, the ratio
\beqs \frac{(\tilde{t}_{ij}^o)_r - s_r}{(\tilde{t}_{ij}^o)_{r'} - s_{r'}} \eeqs is independent of $o, i,j$, whenever the denominator is nonzero; otherwise the numerator is zero as well.

We will need a lower bound on \beq \left(w_0, w_1, w_2\right) \cdot \left(k_0(k_0 + \ell_0), - k_0\ell_0, \ell_0(k_0 + \ell_0)\right).\eeq
Such a lower bound can be expressed using the discriminant. 
Suppose that $|\ell_0|$ is at least as large as $|k_0|$ and hence nonzero (at least one of $k_0$ and $\ell_0$ must be nonzero). We set $x_0 = \frac{k_0}{\ell_0}$.
Then, 
\beqs \left(w_0, w_1, w_2\right) \cdot \left(k_0(k_0 + \ell_0), - k_0\ell_0, \ell_0(k_0 + \ell_0)\right) & =  & 
\ell_0^2\left(w_0x_0^2 + (w_0 - w_1 + w_2) x_0 + w_2\right)\\
& = & \ell_0^2 w_0\left(x_0 + \frac{w_0 - w_1 + w_2}{2w_0}\right)^2\\ 
& + & \ell_0^2w_0\left(\frac{w_2}{w_0} - \left(\frac{w_0 - w_1 + w_2}{2w_0}\right)^2\right)\\
& \geq & \frac{\ell_0^2}{w_0}\left(w_2w_0 - \left(\frac{w_0 - w_1 + w_2}{2}\right)^2\right).\eeqs
We thus have an lower bound on $|n^2 \nabla \f \cdot (\tilde{t}_{ij}^o(g) - s)|$ of $$\left|4\pi^2\left(w_0^{-1} \left(w_2w_0 - \left(\frac{w_0 - w_1 + w_2}{2}\right)^2\right)(k_0^2 + \ell_0^2) + \frac{C\max(|k_0|, |\ell_0|)^3}{n}\right) \Re(\la \hat g)(o_1 + i n_1, o_2 + j n_1)\right|.$$
This lower bound can be rewritten as 
\beqs \Omega\left(\left|\left( \left(2\left(w_0 w_1 + w_1 w_2 + w_2 w_0\right) - w_0^2 - w_1^2 - w_2^2\right)(k_0^2 + \ell_0^2) \right)\Re(\la \hat g)(o_1 + i n_1, o_2 + j n_1)\right|\right),\eeqs since $\frac{1}{w_0}$ is bounded from below by Lemma~\ref{lem:ess}, and $$\max(|k_0|, |\ell_0|) = o(n).$$ 

Thus, denoting  $ \left(2\left(w_0 w_1 + w_1 w_2 + w_2 w_0\right) - w_0^2 - w_1^2 - w_2^2\right)$ by $\Lambda$, we finally get 
\beq |n^2 \nabla \f \cdot (\tilde{t}_{ij}^o(g) - s)| & \geq &   \Omega\left(|(\Lambda + o(1))(k_0^2 + \ell_0^2)\Re(\la \hat g)(o_1 + i n_1, o_2 + j n_1)|\right)\nonumber\\
& \geq & c|\Re(\la \hat g)(o_1 + i n_1, o_2 + j n_1)|.\lab{eq:sub-3.45}\eeq

The last step used the fact that $$k_0^2 + \ell_0^2 \geq 1. $$

Let us multiply $s$ by a suitable scalar, and henceforth assume that $\f(s) = 1.$


By Lemma~\ref{lem:2.13},
\beqs \|f - g\|_{\dot{\CC}^1} \leq Cs_2 \sqrt{\eps_{0.5}} n.\eeqs

Also, by Lemma~\ref{lem:3.4}, we see that
if $f$ satisfies 
$$\|D_r( f)\|_\infty < \check{C} s_r \log n,$$ for $r \in \{0, 1, 2\}$ then, keeping in mind from Definition~\ref{def:3.8} that 
 $M = (s_2 \sqrt{\eps_{0.5}}  n)^{-1}$ and $\check{C} = \frac{\sqrt{\eps_{0.5}}n}{\log n},$ 

\beq|t_{ij}^o(f) - \tilde{t}_{ij}^o(g)| & < & Cn_1^{-1} \left(\|f - g\|_{\dot{\CC}^1}  + \check{C} s_2 \log n + M^{-1}\right)\nonumber\\
                                                            & < &  Cn_1^{-1} \left( Cs_2 \sqrt{\eps_{0.5}} n  + \check{C} s_2 \log n + M^{-1}\right)\nonumber\\
& < & C \left(\frac{\sqrt{\eps_{0.5}}}{\eps_1 }\right).\lab{eq:comp2}\eeq

Recall from Lemma~\ref{lem:3.7} that :
\beqs \left|(\tilde{Q}_n(\tilde{\b} + o) - f)\right| & \leq & n^{\frac{4n_2^2}{n_1^2}}\prod\limits_{1 \leq i, j \leq \frac{n_2 }{n_1}}   |P_{n_1}(t_{ij}^o)|.
\eeqs
In view of (\ref{eq:comp2}) and Corollary~\ref{cor:lip} (which states that  $|\mathbf{f}_n(s) - \mathbf{f}_n(t)| < (2e + \eps)|s - t|$), we therefore have

\beq\left|(\tilde{Q}_n(\tilde{\b} + o) - f)\right| & \leq & n^{\frac{4n_2^2}{n_1^2}}\exp \sum\limits_{1 \leq i, j \leq \frac{n_2 }{n_1}}(n_1^2 - 1)\ln |\mathbf{f}_{n_1}(t_{ij}^o(f))|\\
& \leq & n^{\frac{4n_2^2}{n_1^2}}\exp \sum\limits_{1 \leq i, j \leq \frac{n_2 }{n_1}}(n_1^2 - 1)\ln \left(\mathbf{f}_{n_1}(\tilde{t}_{ij}^o(g)) + C\left(\frac{\sqrt{\eps_{0.5}}}{\eps_1 }\right)\right)\\
& \leq & n^{\frac{4n_2^2}{n_1^2}}\exp \sum\limits_{1 \leq i, j \leq \frac{n_2 }{n_1}}(n_1^2 - 1)\ln \left(\mathbf{f}(\tilde{t}_{ij}^o(g)) + C\left(\frac{\log n_1}{n_1}\right) + C\left(\frac{\sqrt{\eps_{0.5}}}{\eps_1 }\right)\right).\nonumber\\
\lab{eq:from-lem-3.6-2}
\eeq
 The last step above follows from Corollary~\ref{cor:...} which relates $\f_{n_1}$ and $\f$.
We simplify \beqs \ln \left(\mathbf{f}(\tilde{t}_{ij}^o(g)) +  C\left(\frac{\log n_1}{n_1}\right) + C\left(\frac{\sqrt{\eps_{0.5}}}{\eps_1 }\right)\right)\eeqs
further as follows.
\beqs
\ln \left(\mathbf{f}(\tilde{t}_{ij}^o(g)) +  C\left(\frac{\log n_1}{n_1}\right) + C\left(\frac{\sqrt{\eps_{0.5}}}{\eps_1 }\right)\right) & \leq & 
\ln \left(\mathbf{f}(\tilde{t}_{ij}^o(g))\right)  + C\left(\frac{\sqrt{\eps_{0.5}}}{\eps_1 }\right).\eeqs
By the concavity of $\f$, this is in turn less or equal to \beqs \ln \left(1 +  (\tilde{t}_{ij}^o(g) - s) \cdot \nabla \mathbf{f}(s)\right) + C\left(\frac{\sqrt{\eps_{0.5}}}{\eps_1 }\right).
\eeqs
Since $|(\tilde{t}_{ij}^o(g) - s) \cdot \nabla \mathbf{f}(s)|$ is $o_n(1)$, we can use a partial Taylor expansion to obtain
\beqs \ln \left(1 +  (\tilde{t}_{ij}^o(g) - s) \cdot \nabla \mathbf{f}(s)\right) + C\left(\frac{\sqrt{\eps_{0.5}}}{\eps_1 }\right) & \leq & \\
 (\tilde{t}_{ij}^o(g) - s) \cdot \nabla \mathbf{f}(s)) - \frac{|(\tilde{t}_{ij}^o(g) - s) \cdot \nabla \mathbf{f}(s))|^2}{3} + C\left(\frac{\sqrt{\eps_{0.5}}}{\eps_1 }\right).
\eeqs

Thus,

\beqs \sum\limits_{1 \leq i, j \leq \frac{n_2 }{n_1}}(n_1^2 - 1)\left(\ln \mathbf{f}(\tilde{t}_{ij}^o(g)) + C\left(\frac{\log n_1}{n_1}\right) + C\left(\frac{\sqrt{\eps_{0.5}}}{\eps_1 }\right)\right) \eeqs is less or equal to \beq
\sum\limits_{1 \leq i, j \leq \frac{n_2 }{n_1}}(n_1^2  - 1)\left( 
 (\tilde{t}_{ij}^o(g) - s) \cdot \nabla \mathbf{f}(s)) - \frac{|(\tilde{t}_{ij}^o(g) - s) \cdot \nabla \mathbf{f}(s))|^2}{3} + C\left(\frac{\log n_1}{n_1}\right) + C\left(\frac{\sqrt{\eps_{0.5}}}{\eps_1 }\right)\right)
\lab{eq:last2?}.\lab{eq:f-n}\eeq
For a given $g$ and $f$, we would like to guarantee the existence of an offset $o$ and a corresponding polytope 
$(\tilde{Q}_n(\tilde{\b} + o) - f)$ such that the volume bound given by (\ref{eq:from-lem-3.6-2}) is good enough for our purposes.
To do this, it suffices to show that there is an offset $o$ such that \ref{eq:f-n} is fairly negative. Recall that $g$ is the real part of a scaled complex exponential and is thus very well behaved.
The expectation of (\ref{eq:f-n}) with respect to a uniformly random offset $o$ is 
\beqs
\sum\limits_{1 \leq i, j \leq \frac{n_2 }{n_1}}(n_1^2  - 1)\left( (-1)\E_{o}  \frac{|(\tilde{t}_{ij}^o(g) - s) \cdot \nabla \mathbf{f}(s))|^2}{3} + C\left(\frac{\log n_1}{n_1}\right) + C\left(\frac{\sqrt{\eps_{0.5}}}{\eps_1 }\right)\right).\eeqs
Using (\ref{eq:2.67}) and (\ref{eq:sub-3.45}), we obtain the bound \beqs (-n^4)\E_{o}  \frac{|(\tilde{t}_{ij}^o(g) - s) \cdot \nabla \mathbf{f}(s))|^2}{3} & < & (- c)\E_o \Re(\la \hat g)^2(o_1 + i n_1, o_2 + j n_1)\\
& < & (-c )|\theta_{k_0\ell_0}|^2\\
& < &  \frac{(-c)\eps_0^2 n^4}{C_2 (\log \de^{-1})}.\eeqs

Thus, there is an offset $o$ such that (\ref{eq:f-n}) is less than $$(-cn^2) \left(\frac{\eps_0^2}{C_2 (\log \de^{-1})} - C\left(\frac{\sqrt{\eps_{0.5}}}{\eps_1 }\right)\right).$$
For this value of $o$, we see that the following is true.
\beq\left|(\tilde{Q}_n(\tilde{\b} + o) - f)\right| & \leq & n^{\frac{4n_2^2}{n_1^2}}\exp \sum\limits_{1 \leq i, j \leq \frac{n_2 }{n_1}}(n_1^2 - 1)\ln |\mathbf{f}_{n_1}(t_{ij}^o(f))|\\
& \leq & n^{\frac{4n_2^2}{n_1^2}}\exp \sum\limits_{1 \leq i, j \leq \frac{n_2 }{n_1}}(n_1^2 - 1)\ln \left(\mathbf{f}(\tilde{t}_{ij}^o(g)) + C\left(\frac{\log n_1}{n_1}\right) + C\left(\frac{\sqrt{\eps_{0.5}}}{\eps_1 }\right)\right)\nonumber\\
& \leq & \exp \left(
(-cn^2) \left(- \frac{C\log n}{n^2 \eps_1^2} + \frac{\eps_0^2}{C_2 (\log \de^{-1})} - C\left(\frac{\eps_1 \log n_1}{n} + \frac{\sqrt{\eps_{0.5}}}{\eps_1 }\right)\right)\right).\nonumber\\\lab{eq:master1}
\eeq
Let $\eps$ be an  small  positive constant depending only on $s$. We use $a \lesssim b$ to mean that $a < n^{- O(\eps)} b.$
We also have the constraint that $n_1 = o(\max(k_0, \ell_0)^{-1} n),$ since the wavelengths of the complex exponential needs to be large compared to $n_1$.  
To ensure this, it suffices to have
$\frac{C_2n_1^2\log \de^{-1}}{cn^2 \eps_0^2} = o(1),$ which is satisfied if one assumes that  \beq \de = \exp(- n^{\eps})\eeq and \beq \eps_1 \lesssim \eps_0.\eeq
We now write down a sufficient family of constraints needed to make  (\ref{eq:master1}) less than $o(1)$. 
\ben
\item $ \frac{1}{n^2 \eps_1^2} \lesssim \eps_0^2$.
\item $\frac{\eps_1}{n} \lesssim \eps_0^2.$
\item $\frac{\sqrt{\eps_{0.5}}}{\eps_1} \lesssim \eps_0^2.$
\een
Additionally, due to our use of Bronshtein's theorem, we must ensure that
\beq \eps_{0.5}^{-1} \lesssim n^2\eps_0^2, \eeq

and the number of covering polytopes for each cube is bounded above using   Lemma~\ref{lem:polytope_number}, by $n^{\frac{Cn}{\eps_1}} \leq O\left(\exp\left(C n\eps_1^{-1}\log n\right) \right)$ polytopes.

Therefore we must also ensure that \beq \frac{n}{\eps_1} \lesssim n^2 \eps_0^2.\eeq
These conditions can be satisfied as follows. 

Set $$\eps_1 := n^{-(\frac{1}{4})}.$$ Set $$\eps_{0.5} := n^{-\frac{3}{2}}.$$ Finally, we set $$\eps_0 = n^{-\left({\frac{1}{4} - \eps}\right)}.$$

Let $\A_1$ be the subset of $P_n(s)$ consisting of all those $f$ for which $\|f\|_{L_2^2} \geq C_2n\ln \de^{-1}$ for $\de = \exp( - n^\eps),$ which by  Lemma~\ref{lem:2.11}, has measure at most $\exp(-n^\eps)|P_n(s)|$. 
Let $\A_2$ be the subset of $P_n(s)$ consisting of all those $f$ such that for each $r\in \{0, 1, 2\}$, $\|D_r(f)\|_\infty > \check{C}\log n$, which by Lemma~\ref{lem:2.125} has measure at most $n^{-c\check{C}+2}|P_n(s)|$.

Using Lemma~\ref{lem:polytope_number}, which provides an upper bound on the number of covering polytopes $\tilde{Q}(\b, s, x)$, we have the following for all sufficiently large $n$. There is a positive constant $\eps_2$ depending only on $s$ such that
\beqs \left|\left(P_n(s) \setminus (\A_1\cup \A_2)\right) \cap B_\infty(g, \eps_{0.5}n^2)\right| < \exp(-\eps_2 n^{\frac{5 - \eps}{4}})|P_n(s)|.\eeqs
Using Bronshtein's upper bound of (\ref{eq:Br}) and Lemma~\ref{lem:2.14} with $K = P_n(s)\setminus(\A_1 \cup \A_2)$, we see that 
\beqs \left|P_n(s) \setminus (\A_1\cup \A_2 \cup B_\infty(0, \eps_0n^2)) \right| <\exp( C_{br}\eps_{0.5}^{-1})\exp(- {\eps_2 n^{\frac{5 - \eps}{4}}})|P_n(s)|.\eeqs
Therefore
\beqs |P_n(s)\setminus B_\infty\left(0, \eps_0 n^2\right)| & \leq & \left(\left(n^{-c\check{C}+2}\right) + C \exp(- n^{\frac{5 - \eps}{4}}) + \de\right)|P_n(s)|.\eeqs

It follows that the probability measure of all $g \in P_n(s)$ such that $\|g\|_2 \geq \eps_0 n^3$ is $o(1)$. By the logconcavity of the distribution of $g(v)$ for a fixed $v \in V(\T_n)$ and uniformly random $g$ from $P_n(s)$, it follows that $\E[\|g\|_2^2] \leq C \eps_0^2 n^6.$ By Lemma~\ref{lem:5.8-Nov-2020}, 
$$\p\left[\|g\|_\infty >  \left(\frac{\a \log n}{n}\right) \sqrt{\E \|g\|_2^2}\right] < n^{-c \a}$$ and
our proof is complete.\end{proof}

\begin{lem}\lab{lem:7.6}

Suppose $0 < e_0 = e_1 \leq e_2.$ then denoting $(e_0, e_1, e_2)$ by $e$, we have $(\wn_0(e), \wn_1(e), \wn_2(e)) \in \C.$

\end{lem}

\begin{proof}
By the anisotropic isoperimetric inequality (\ref{eq:2.2}), applied to $K = P_n(k)$ and $E = P_n(e)$, 
we have 
\beq S_K(E) S_E(K) \geq (n^2 - 1)^2|K| |E|.\eeq Let $k = (2, 2, 2)$.

Then,
\beqs \left(\frac{(n^2 - 1)|K|(e_0 + e_1 + e_2)}{3}\right) \sum_i \wn_i(e) \geq (n^2 - 1)|K| \sum_i \wn_i(e) e_i.\eeqs

This implies that \beq  \frac{\wn_0(e)+ \wn_1(e) + \wn_2(e)}{3} \geq \frac{ \wn_0(e) e_0 + \wn_1(e) e_1 + \wn_2(e) e_2}{e_0 + e_1 + e_2}.\lab{eq:4.2}\eeq
Observe that, $e_0 = e_1 \leq e_2$ and so by symmetry, $\wn_0(e) = \wn_1(e)$.
Thus, (\ref{eq:4.2}) implies that $\wn_2(e) \leq \wn_1(e) = \wn_0(e).$ Putting this together with Lemma~\ref{lem:ess} shows that $(\wn_0(e), \wn_1(e), \wn_2(e)) \in \C.$
\end{proof}
Thus, if $s_0 = s_1 \leq s_2$, it is always possible to choose a superdifferential $w(s)$ such that 
$w(s) \in \C$ by taking a subsequential limit of the  sequence $(w^{(n)})_{n\in \N}$. Therefore, for such $s$, the conclusion of Theorem~\ref{thm:2} holds, namely,
for any $\eps > 0$, 
we have  $$\lim_{n \ra 0} \p\left[\|g\|_\infty > n^{\frac{7}{4} + \eps}\right] = 0$$  when $g$ is randomly sampled from $P_n(s)$.

\subsection{Concentration of random honeycombs with periodic boundary conditions}\lab{ssec:honey}
As mentioned in the first section of this paper,
one obtains a random honeycomb  from a random hive by mapping the gradient of the hive on each of the unit equilateral triangles to a point in $\R^2$. These points become the vertices of the honeycomb.
Let us consider an infinite random hive with a periodic Hessian that averages to $s$, (which thus is,  after suitable transformation, an element of $P_n(s)$) and the corresponding periodic honeycomb $\tau$, and rescale the torus $\T_n$ so that the corresponding fundamental domain has a unit side length. Let us also scale down the hive by a factor of $n^2$, and interpolate it, in a piecewise linear fashion to obtain a Lipschitz function $h$ fom the unit torus $\T$ to the reals. Theorem~\ref{thm:2} from the previous subsection shows that the probability that this semiconcave Lipschitz function $h$ differs from the $0$ function  by more than $n^{-\frac{1}{4} + \eps}$ in $\ell_\infty$ tends to $0$ as $n \ra \infty.$ This implies that with probability tending to $1$, at no point on $\T$ does the gradient of $h$ have a length (measured using the Euclidean norm in $\R^2$) that is more than $O(n^{-\frac{1}{8} + \frac{\eps}2}),$ due to semiconcavity. In other words, with probability tending to $1$, every vertex of the random periodic honeycomb $\tau$ is within $O(n^{\frac{7}{8} + \frac{\eps}{2}})$ of the position of the corresponding  vertex for a honeycomb corresponding to a quadratic function with Hessian $s$. On the other hand, there exist honeycombs with a displacement  $\Omega(n),$ because, by Lemma~\ref{lem:diameter}, it is possible to have discrete gradients this large in some semiconcave functions belonging to $P_n(s)$.

\end{document}